\documentclass{article}
\usepackage{graphicx}
\usepackage[left=2cm]{geometry}
\usepackage[page]{appendix}
\usepackage{amsmath}
\usepackage{amssymb}
\usepackage{amsthm}
\usepackage{stmaryrd}
\usepackage{txfonts}
\usepackage{pstricks,pst-node}
\usepackage{subfloat}
\usepackage{caption}
\usepackage{subcaption}
\usepackage{url}
\usepackage{multirow}
\usepackage{enumerate}
\usepackage{thmtools}

\newtheorem{lemma}{Lemma}[section]
\newtheorem{theorem}[lemma]{Theorem}
\newtheorem{proposition}[lemma]{Proposition}

\newtheorem{definition}[lemma]{Definition}
\newtheorem{assumption}[lemma]{Assumption}

\newtheorem*{theorem*}{Theorem}

\theoremstyle{definition}

\newtheorem{examples}[lemma]{Examples}

\theoremstyle{remark}

\begin{document}

\title{Consistency of Ranking Estimators} 
\author{Toby~Kenney}

\maketitle

\begin{abstract}
The ranking problem is to order a collection of units by some
unobserved parameter, based on observations from the associated
distribution. This problem arises naturally in a number of contexts,
such as business, where we may want to rank potential projects by
profitability; or science, where we may want to rank predictors
potentially associated with some trait by the strength of the
association. This approach provides a valuable alternative to the
sparsity framework often used with big data. Most approaches to this
problem are empirical Bayesian, where we use the data to estimate the
hyperparameters of the prior distribution, then use that distribution
to estimate the unobserved parameter values. There are a number of
different approaches to this problem, based on different loss
functions for mis-ranking units. Despite the number of papers developing
methods for this problem, there is no work on the consistency of these
methods. In this paper, we develop a general framework for consistency
of empirical Bayesian ranking methods, which includes nearly all
commonly used methods. We then determine conditions under which
consistency holds. Given that little work has been done on selection
of prior distribution, and that the loss functions developed are not
strongly motivated, we consider the case where both of these are
misspecified. We show that provided the loss function is reasonable;
the prior distribution is not too light-tailed; and the error in measuring
each unit converges to zero at a fast enough rate compared with the
number of units (which is assumed to increase to infinity); all
ranking methods are consistent.
\end{abstract}

\subsection*{keywords}
Consistency; Ranking; Empirical Bayes

\section{Introduction}
The ranking problem is to order a collection of units. The field is
divided into two main cases. The first is the latent variable ranking,
where each unit $u_i$ has some tangible latent parameter $\theta_i$
and we want to rank them according to the values of this parameter.
The more general case is loss-based ranking, where for each suggested
ranking, we are given a loss score, which follows a certain
distribution, and our objective is to select the ranking that
minimises this loss score. The latent variable ranking problem is the
focus of the current paper.

We formalise the problem as follows. A latent-variable ranking problem
consists of a collection ${\mathcal U}=\{u_1,u_2,\ldots,u_p\}$ of
units with unobserved parameters $\theta_i$. For each unit, we have a
point estimate $x_i$ for $\theta_i$. We assume that $x_i$ has a known
error distribution with mean $\theta_i$ and known standard deviation
$\sigma_i$. The objective is to assign a ranking to the units, namely
a bijection $\rho:{\mathcal U}\rightarrow \{1,2,\ldots,p\}$ such that
$\rho(u_i)<\rho(u_j)$ whenever $\theta_i>\theta_j$. 

This problem was first studied as a formal statistical problem by
Bechhofer (1954) and by Gupta (1956). There are many examples of this
problem arising naturally: we may wish to rank genes by the risk they
cause of a particular condition; we may wish to rank sportsmen by
their success-rate at particular standardised trials; we may wish to
rank business opportunities by the profit they will generate.
Typically, for each unit we wish to rank, we will have some data on
the associated feature, but will not know the true value of that
feature. Based on our data, we will have a point estimate for the
feature, and an associated error distribution. The amount of data we
might have for different units can vary wildly, meaning that the
associated error distributions can be very different for different
units. This means that the na\"{i}ve rank function that ranks $u_i$
before $u_j$ if $x_i>x_j$ will rank many false positives among the
highest ranked units. On the other hand, if the main aim is to avoid
false positives, a testing-based approach ranks the units according to
the $p$-values of hypothesis tests for a given null value. This
approach will often give high ranks to units on which we have
collected most data, even if they are not the most important. This may
lead to neglecting some units which have much higher underlying value,
but for which we have less data. A third method, posterior expected
rank (Laird and Louis, 1989) calculates the probability for each
pairwise rank based on the error distribution, then calculates the
expected rank for each unit as the total probability that it is ranked
after each other unit.

Other approaches to the problem mainly take an empirical Bayesian
approach. They assume that $\theta_i$ follow some prior distribution;
estimate this underlying distribution from all the data points; then
for unit $u_i$, estimate the posterior distribution for $\theta_i$
using the estimated distribution of the prior. Ranking is then
performed using the Bayes method for a chosen loss function. There are
a range of different methods based on different loss functions. Indeed
the three methods in the preceding paragraph correspond to three
different loss functions in a frequentist framework: the na\"{i}ve
approach maximises the expected total value of top units; the testing
approach maximises the expected value of a zero-one loss which is one
if $\rho(u_i)<\rho(u_j)$ and $\theta_i=\theta_{\textrm{Null}}$ and
$\theta_j>\theta_{\textrm{Null}}$. Posterior expected rank minimises
the expected total number of misranked pairs.  The same loss functions
could easily be applied in the Bayesian setting. Maximisation of
expected total value of top units is achieved by posterior mean
ranking, used for example in Aitkin and Longford (1986). This is shown
in Gupta and Hsiao (1983), though the statement is not made very
explicitly, and includes unnecessary hypotheses.  Another recently
developed loss function is the $r$-values method (Henderson and
Newton, 2016), which corresponds to a loss function being the sum of
absolute differences between estimated rank and true rank. This method
is based on the underlying distribution of both $\theta$ and
$\sigma$. A range of other loss functions have also been considered,
for example, Lin {\em et al}.  (2006) summarise a range of choices of
loss function.

While questions about the consistency of various approaches to the
loss-based ranking problem have led to interesting research (Duchi,
Mackey \& Jordan 2010, Duchi, Mackey \& Jordan 2013). Consistency
questions have not been studied to the same extent in the latent
variable ranking problem. This may be because in the classical problem
where the units are fixed and the estimates become more accurate, it
is obvious that the majority of reasonable methods are
consistent. However, in the era of big data, it is becoming more
common to study consistency not just as sample size increases (which
for latent variable ranking corresponds to the point estimates
becoming more accurate) but also as number of variables (in our case
units) increases. Studying this consistency issue is the focus of the
current paper. Typically in the literature, choice of prior
distribution and loss function are arbitrary. We therefore consider
consistency in the case where these are misspecified.
An additional challenge to studying consistency for the ranking
problem is how consistency should be defined for this problem, since
the quantity being estimated is a ranking, rather than the value of
some parameter.


The structure of this paper is as follows: In
Section~\ref{ConsistencyFramework}, we formalise the assumptions we
need on conditional distributions, loss functions and prior
distributions, and deduce a key result bounding the expected posterior
pairwise loss function. In Section~\ref{Proof}, we use this result to
prove consistency of ranking methods. In Section~\ref{Examples}, we
demonstrate through examples that if the prior distribution is too
light-tailed, then ranking methods can be inconsistent.

\section{Framework for Ranking Estimators}\label{ConsistencyFramework}

In this section, we outline the regularity conditions that we impose
on the problem and on the ranking estimators. We formally state the
problem as follows. We have a fixed sequence
$\Theta_1,\Theta_2,\ldots$ of true values. These values are
i.i.d. from some continuous true prior distribution. Now at every
stage $p$, we have point estimators $X_{1p},\ldots,X_{pp}$ for
$\Theta_1,\ldots,\Theta_p$. Every $X_{ip}$ is drawn independently
from a 2-parameter distribution $g(\Theta_i,\sigma_{ip})$, where
$\Theta_i$ is the mean and $\sigma_{ip}$ is the variance. We assume
that the $\sigma_{ip}$ are i.i.d. from some distribution
$h_p(\sigma)$ with finite mean and variance. We will assume that the
conditional distribution of $X_{ip}$ given $\Theta_i$ and
$\sigma_{ip}$ is known perfectly. In cases where it is uncertain, we
can often use model averaging to calculate a different form for the
conditional distribution. For example, if the conditional distribution
is normal, but the variance $\sigma_i$ is uncertain, and has posterior
distribution a scaled inverse chi-square distribution, then model
averaging leads us to use a Student's $t$ distribution instead.

The majority of ranking estimators are empirical Bayesian in nature,
and even the methods that are not can be recast in empirical Bayesian
form by considering improper prior distributions. A ranking method
consists of two parts: a prior distribution $\pi$ for $\Theta$ and a
sequence of loss functions $(L_p:{\mathbb R}^p\rightarrow{\mathbb
  R})_{p=1,2,\ldots}$ satisfying, for any
$\theta_1\geqslant\cdots\geqslant\theta_p$,
$L_p(\theta_1,\ldots,\theta_p)=0$. We will consider the misspecified
case where $\pi$ is different from the true prior. Given these two
pieces of data, the ranking estimator will be the permutation
$\hat{\rho}_p({\mathbf x})\in S\!_p$ which minimises the posterior
expected loss
$${\mathcal L}_{{\mathbf x},\pi,L_p}(\rho)={\mathbb
  E}_{\Theta|{\mathbf
    x}}\left(L_p\left(\Theta_{\rho^{-1}(1)},\Theta_{\rho^{-1}(2)},\ldots,\Theta_{\rho^{-1}(p)}\right)\right)$$
Where the posterior distribution $\Theta|x$ is calculated in the usual
way with density proportional to
$\pi(\theta)L_{x,\sigma}(\theta)$, where $L_{x,\sigma}(\theta)$ is the
likelihood of $\theta$. The regularity conditions on the
ranking estimator can therefore be divided into conditions on the
prior and conditions on the loss.

\subsection{Regularity Conditions on the Problem and on the Prior}

One of the key requirements for consistency of ranking methods is that
the posterior distribution of $\Theta_i$ should not be too far from
the observed value $X_{ip}$. In this section, we define natural
conditions on the conditional distribution of $X_{ip}$ given
$\Theta_i$ and $\sigma_{ip}$, and on the prior distribution of
$\Theta$, and show that under these conditions, the posterior mean is
close to the observed value, and the posterior variance is bounded.

Our first assumption is on the conditional distribution:

\begin{assumption}\label{RegularityCondition}
  There is an integrable likelihood function $L(X;\theta,\sigma)$ and
  a finite constant $K$
  satisfying
  $$\int_{-\infty}^\infty (K\sigma^2-(\theta-x)^2)L(x;\theta,\sigma)\,d\theta>0$$
\end{assumption}

The idea here is that the likelihood function is proportional to the
density function of a continuous distribution, and the mean of this
distribution is in the interval
$\left[x-\sqrt{K}\sigma,x+\sqrt{K}\sigma\right]$, and the variance is
bounded by $K\sigma^2$. The definition of the likelihood function can
become tricky for some distributions of $X$, so we will restate this
assumption more formally as:

\addtocounter{lemma}{-1}
\renewcommand{\theassumption}{\arabic{section}.\arabic{lemma}$'$}
\begin{assumption}\label{RegularityConditionFormal}
For all sufficiently small $\delta>0$, and for all $x$,
  $$\int_{-\infty}^\infty (K\sigma^2-(\theta-x)^2)P\left(X\in[x-\delta,x+\delta]\middle|\Theta=\theta,\sigma\right)\,d\theta>0$$
\end{assumption}
\renewcommand{\theassumption}{\arabic{section}.\arabic{lemma}}

We will define the random variable $\Theta_{L(x,\sigma)}$ to follow the
posterior distribution of $\Theta$ under an improper uniform
prior. That is, for continuous $\Theta_{L(x,\sigma)}$, the density is
proportional to the likelihood
$L(x;\theta,\sigma)$. Assumption~\ref{RegularityCondition} states that
$\Theta_{L(x,\sigma)}$ has mean close to $x$ and variance less than $K\sigma^2$.

%
%
%
\begin{lemma}\label{AlternativeRegularityCondition}
  If there is a constant $\epsilon>0$ such that for any $\delta>0$,
  and any $x$ and $x'$,
  $$\inf_{\stackrel{\theta\in[x-\epsilon,x+\epsilon]}{\scriptscriptstyle\theta'\in[x'-\epsilon,x'+\epsilon]}}P_{\theta,\sigma}(X\in
  [x'-\delta,x'+\delta])-\epsilon
  P_{\theta',\sigma}(X\in[x-\delta,x+\delta])>0$$
then Assumption~\ref{RegularityConditionFormal} holds.
\end{lemma}

\begin{proof}
  W.l.o.g., we may assume $\epsilon<1$. We need to show that for all
  sufficiently small $\delta$, we have
  $$\int_{-\infty}^\infty (K\sigma^2-(\theta-x)^2)P(X\in[x-\delta,x+\delta]|\theta,\sigma)\,d\theta>0$$
  We have that ${\mathbb
    E}\left((X-\theta)^2|\theta,\sigma\right)<\sigma^2$. It follows
  that

  \begin{align*}
\int_{-\infty}^\infty
(K\sigma^2-(\theta-x)^2)P(X\in[x-\delta,x+\delta]|\Theta=\theta)\,d\theta&=\int_{-\infty}^\infty
K\sigma^2P(X\in[x-\delta,x+\delta]|\Theta=\theta)\,d\theta-\int_{-\infty}^\infty
(\theta-x)^2P(X\in[x-\delta,x+\delta]|\Theta=\theta)\,d\theta\\
&\geqslant\int_{-\infty}^\infty
K\sigma^2P(X\in[x-\delta,x+\delta]|\Theta=\theta)\,d\theta-\frac{1}{\epsilon}\int_{-\infty}^\infty
(\theta-x)^2P(X\in[\theta-\delta,\theta+\delta]|\Theta=x)\,d\theta\\
&\geqslant \int_{-\infty}^{\infty}
K\sigma^2P(X\in[x-\delta,x+\delta]|\Theta=\theta)\,d\theta-\frac{2\delta}{\epsilon}{\mathbb
  E}((X-x)^2)\\
&\geqslant \epsilon
K\sigma^2\int_{-\infty}^{\infty}P(X\in[\theta-\delta,\theta+\delta]|\Theta=x)\,d\theta-\frac{2\delta}{\epsilon}{\mathbb
  E}((X-x)^2)\\
&=2\delta\epsilon
K\sigma^2-\frac{2\delta\sigma^2}{\epsilon}
  \end{align*}
We can therefore ensure the result by setting
$K=\frac{\sigma^2}{\epsilon^2}$.   
\end{proof}

The condition in Lemma~\ref{AlternativeRegularityCondition} is like a
weakened form of translation-invariance. If $\theta$ is a location
parameter, and the error distribution is symmetric, then the condition
will hold for any $\epsilon<1$.

To ensure that the posterior mean and variance are close to the
observed value, we will also require the following conditions on the
prior distribution $\pi$:

\begin{enumerate}

\item $\pi$ is a continuous distribution.

\item $\pi$ is {\em quasiunimodal}, meaning there is some {\em
  quasimode} $m$ and some $\epsilon>0$ such that for any
  $x_2\leqslant x_1\leqslant m$ and for any $m\leqslant x_1\leqslant
  x_2$, we have $\pi(x_1)\geqslant\epsilon\pi(x_2)$. We see that this is a
  generalisation of unimodality --- $\pi$ is unimodal if it is
  quasiunimodal with $\epsilon=1$.
  
\end{enumerate}

\begin{lemma}
  If $m$ is a quasimode for a continuous distribution $\pi$, then so
  is any other $x$ in the support of $\pi$.
\end{lemma}

\begin{proof}
W.l.o.g., let $x_2\leqslant x_1\leqslant x$. We want to find some
$c>0$ such that $\pi(x_1)\geqslant c\pi(x_2)$. We always have
$m\leqslant m\leqslant x_2$, or $m\geqslant m\geqslant x_2$ so $\pi(m)\geqslant\epsilon\pi(x_2)$.
There are two cases to consider:
\begin{enumerate}
\item $m\leqslant x_1$:  In this case,
  $\pi(x_1)\geqslant\epsilon\pi(x)$, so that
  $$\pi(x_1)\geqslant \epsilon\pi(x)\left(\frac{\epsilon\pi(x_2)}{\pi(m)}\right)=\left(\frac{\epsilon^2\pi(x)}{\pi(m)}\right)\pi(x_2)$$

\item $m\geqslant x_1$: In this case,
  $$\pi(x_1)\geqslant\epsilon\pi(x_2)$$
\end{enumerate}
Since $\pi(m)\geqslant\epsilon\pi(x)$, we have
$\frac{\epsilon^2\pi(x)}{\pi(m)}\leqslant \epsilon$.
Thus, $x$ is also a quasimode with $c=\frac{\epsilon^2\pi(x)}{\pi(m)}$.
\end{proof}

The reason these properties are necessary, is that if both the prior
and conditional distribution are allowed to be discrete, then by
controlling the support of both prior and conditional distribution, we
can cause pathological behaviour in the posterior distribution.

Finally, we need to control the tail of the prior distribution, since if
the prior is too light-tailed, the posterior distribution will be
dominated by the prior, rather than the observed data.

\begin{definition}
  A prior distribution $\Theta$ with density $\pi(\theta)$ is
  {\em tail-dominating} if there are constants $r$ and $s$, such that for
  any $a>0$ and any $x>a$, we have $${\mathbb
    E}(((\Theta_{L;x,\sigma}-x)^2-a^2)_+)<s\sigma a\pi\left(\frac{ra}{\sigma}\right)$$
  \end{definition}

This will typically hold when the prior distribution is as
heavy-tailed as the conditional distribution, so it will usually hold
for the conjugate prior distribution, but not for more light-tailed priors.

\begin{examples}
For a normal conditional distribution with fixed variance $x\sim
N(\Theta,\sigma^2)$, the likelihood distribution is
$\Theta_{L;x,\sigma}\sim N(x,\sigma^2)$,
so we have
\begin{align*}
{\mathbb E}(((\Theta_{L;x,\sigma}-x)^2-a^2)_+)&=2\int_{x+a}^\infty
\left((\theta-x)^2-a^2\right)\frac{e^{-\frac{(\theta-x)^2}{2\sigma^2}}}{\sqrt{2\pi}\sigma}\,d\theta\\
&=2\int_{a}^\infty
\left(t^2-a^2\right)\frac{e^{-\frac{t^2}{2\sigma^2}}}{\sqrt{2\pi}\sigma}\,dt\\
&\leqslant \frac{2}{\sqrt{2\pi}\sigma}\int_a^\infty t^2e^{-\frac{t^2}{2\sigma^2}}\,dt\\
&=
\frac{2}{\sqrt{2\pi}\sigma}\left(\left[-\sigma^2te^{-\frac{t^2}{2\sigma^2}}\right]_a^\infty+\int_a^\infty
\sigma^2e^{-\frac{t^2}{2\sigma^2}}\,dt\right)\\
&=2\sigma^2\left(a\frac{e^{-\frac{a^2}{2\sigma^2}}}{\sqrt{2\pi}\sigma}+1-\Phi\left(\frac{a}{\sigma}\right)\right)\\
&\leqslant s\sigma a e^{-\frac{1}{2}\left(\frac{a}{\sigma}\right)^2}
\end{align*}
for some $s$. Thus, for any $\tau$, we can set $r=\tau$, to get that a
normal prior is tail-dominating.
\end{examples}

\begin{lemma}\label{PostMeanVar}
If $X$ is a random variable with mean $\Theta$ and variance
$\sigma^2$, satisfying Assumption~\ref{RegularityConditionFormal} and
the prior distribution for $\Theta$ is tail-dominating continuous
quasiunimodal, then there is some $c$ such that the
posterior distribution of $\Theta|x$, given an observation $X=x>0$ has
mean $\mu>x-c(x+1)\sigma$ and variance $s^2<c\sigma^2(x+1)^2$ for any
$\sigma$ such that $\sigma c<1$.
\end{lemma}

\begin{proof}
  The posterior mean is given by
$\mu={\mathbb E}_{\Theta|X}(\Theta)$,  and the posterior variance is given by 
  $s^2={\mathbb E}_{\Theta|X}((\Theta-\mu)^2)$.
  We therefore have
  $s^2+(\mu-x)^2={\mathbb E}_{\Theta|X}((\Theta-x)^2)$
  so it is sufficient to show that
 ${\mathbb E}_{\Theta|X}((\Theta-x)^2-a^2)<0$, where
  $a=c\sigma(x+1)$.

This is equivalent to showing that ${\mathbb E}\left(\left(\left(\Theta_{L(x,\sigma)}-x\right)^2-a^2\right)\pi\left(\Theta_{L(x,\sigma)}\right)\right)<0$.  
We will separately consider the negative part and the positive part.
The negative part contains the interval
$|\Theta-x|<\frac{c\sigma}{2}$, on which $\pi(\theta)>\epsilon\pi(x+c\sigma)$
and $a^2-(\Theta-x)^2>c^2\sigma^2\left(x^2+2x+\frac{3}{4}\right)$.
Therefore
\begin{align*}
{\mathbb E}\left(\left(\left(\Theta_{L(x,\sigma)}-x\right)^2-a^2\right)_-\pi\left(\Theta_{L(x,\sigma)}\right)\right)&>\epsilon\pi(x+c\sigma)\left(c^2\sigma^2\left(x^2+2x+\frac{3}{4}\right)\right)P\left(x-\frac{c}{2}\sigma<\Theta_{L(x,\sigma)}<x+\frac{c}{2}\sigma\right)\\
&>\epsilon\pi(x+c\sigma)\left(c^2\sigma^2\left(x^2+2x+\frac{3}{4}\right)\right)\left(1-\frac{4K}{c^2}\right)\\
\end{align*}

Meanwhile, for the positive part, we have 
\begin{align*}
{\mathbb E}\left(\left(\left(\Theta_{L(x,\sigma)}-x\right)^2-a^2\right)_+\pi\left(\Theta_{L(x,\sigma)}\right)\right)&<\pi(0)\int_{-\infty}^\infty\left((\theta-x)^2-a^2\right)_+L_{x,\sigma}(\theta)\,d\theta\\
&<\pi(0)s\sigma a\pi\left(\frac{ra}{\sigma}\right)\\
&=\pi(0)sc\sigma^2(x+1)\pi\left(rc(x+1)\right)\\
&\leqslant \frac{\pi(0)s c\sigma^2(x+1)\pi(x+1)}{\epsilon}
\end{align*}
where the last line is under the assumption that $rc>1$. Thus, it is sufficient to show that 
  $$\frac{\pi(0)s c\sigma^2(x+1)\pi(x+1)}{\epsilon}<\epsilon\pi(x+c\sigma)\left(c^2\sigma^2\left(x^2+2x+\frac{3}{4}\right)\right)\left(1-\frac{4K}{c^2}\right)$$
Since $\pi$ is quasiunimodal, we have that $\pi(x+c\sigma)>\epsilon
\pi\left(x+1\right)$ as $\sigma c<1$. If we further set
$\frac{c}{\sqrt{K}}>4$, then $\left(1-\frac{4K}{c^2}\right)>\frac{3}{4}$. Therefore, the
required inequality becomes
$$x+1<\frac{9\epsilon^3c}{16s\pi(0)}\left(x^2+2x+\frac{3}{4}\right)$$
Thus it is true for all $x\geqslant 0$ whenever
$c>\frac{64s\pi(0)}{27\epsilon^3}\lor 4K\lor\frac{1}{r}$.
\end{proof}

\subsection{Regularity Conditions on Loss Functions}

For the loss function, we need to ensure both that worse misrankings
incur higher losses, and also that there is consistency between the
loss functions at different stages (that is, between the different
$L_p$). The simplest way to achieve this is using the class of {\em
  additive} loss functions. A loss function is additive if it is given
by
$$L_p(\theta_1,\ldots,\theta_p)=s(p)\sum_{i<j}l(\theta_i,\theta_j)$$
where $s(p)$ is an arbitrary scaling function, and
$l(x,y)$ is a pairwise loss function, which is zero if
$x\geqslant y$ and positive if $x<y$. We will say that $L_p$ is the
additive loss function {\em generated by} $l$.
Note that the scaling function $s(p)$ does not affect the estimated rankings $\hat{\rho}_p$. We will furthermore require the following properties of
the pairwise loss function:

\begin{definition}
  A binary function $l(x,y)$ is {\em restrained} if it satisfies the
  following conditions:

\begin{itemize}

\item $l(x,y)$ is increasing in $y$ and decreasing in $x$.

\item $l$ is Lipschitz with constant $\lambda$ on the set
  $\{(x,y)|x<y\}$.

\item There are constants $a,b>0$ such that for any $x,y$
  satisfying $x<y<x+a$, we have $l(x,y)>b(y-x)$.

\item There is a constant $D$ such that for all $\epsilon>0$,
  $l(x,x+\epsilon)<D+\lambda\epsilon$. 
\end{itemize}

\end{definition}

Many ranking methods in the literature are empirical Bayes for
additive loss functions generated by restrained pairwise loss
functions. However, to improve generality, we extend the class of loss
functions studied by considering equivalent loss functions.

\begin{definition}
Sequences of loss functions $L_p$ and $M_p$ are {\em equivalent} if there are
constants $0<c<u$ such that for all $p$ and for all $x_1,\ldots,x_p$, we have
$cL_p(x_1,\ldots,x_p)<M_p(x_1,\ldots,x_p)<uL_p(x_1,\ldots,x_p)$
\end{definition}

\begin{definition}
  A sequence of loss functions $L_p$ is {\em 
    regular} if it is equivalent to an additive loss function
  generated by a restrained pairwise loss.
\end{definition}

\begin{definition}
  We will refer to the ranking method based on minimising the
  posterior expectation of a loss function ${L}_p$, under a prior
  distribution $\pi$ as the {\em Bayesian ranking} for ${L}$ and
  $\pi$.  A ranking method is {\em standard} if it minimises posterior
  expectation of loss under a regular loss function and a
  tail-dominating quasiunimodal continuous prior distribution.
\end{definition}

\subsection{Common Loss Functions}

It turns out that the above definitions of regular loss functions are
sufficient to cover all commonly used ranking methods. In this
section, we show how existing ranking methods can be cast into this
framework.

\subsubsection{Value Ranking}

When the conditional distribution is translation-invariant, if we use
an improper uniform prior, we have that the posterior mean is equal to
the observed value. This means that when we use a uniform prior for
translation-invariant conditional distribution, value ranking is a
special case of posterior mean ranking.

\subsubsection{Posterior Expected Rank}

Posterior Expected Rank is a positional loss function, in that the
loss for misranking units with true values $\theta_i$ and $\theta_j$
depends upon the true ranks of those units. In particular, the loss is
proportional to the difference between those ranks. However, for any
particular true prior distribution, there is an additive loss function
that will asymptotically approach PER loss. If the true prior
distribution has distribution function $F(\theta)$, then PER
asymptotically minimises the additive loss function generated by
$l(\theta_i,\theta_j)=|F(\theta_j)-F(\theta_i)|$. Under normal
regularity conditions for the distribution, this function is
restrained.

\subsubsection{$p$-value Ranking}

$p$-value ranking is based on improper uniform prior, and additive loss
function generated by $$l(x,y)=\left\{\begin{array}{ll}
1&\textrm{if }x<\theta_0,y\geqslant\theta_0\\
1&\textrm{if }x\leqslant\theta_0,y>\theta_0\\
0&\textrm{otherwise}\end{array}\right.$$
where $\theta_0$ is the null hypothesis value for $\theta$. This loss
function is not restrained, since for $\theta_0<x<y$ we have $l(x,y)=0$.

\subsubsection{Posterior Mean Ranking}

Posterior mean ranking minimises the additive loss function generated by
$$l(x,y)=(y-x)_+$$

This is proved (with some unnecessary hypotheses) in
(Gupta and Hsiao, 1983). However, we present a simpler proof here.

The total loss for a ranking $\rho$ is given by
\begin{align*}
  \sum_{\rho(i)<\rho(j)}l(\theta_i,\theta_j)&=\frac{1}{2}\left(\sum_{\rho(i)<\rho(j)}\left(l(\theta_i,\theta_j)+l(\theta_j,\theta_i)\right)+\sum_{\rho(i)<\rho(j)}\left(l(\theta_i,\theta_j)-l(\theta_j,\theta_i)\right)\right)\\
  &=\frac{1}{2}\left(\sum_{i,j}|\theta_i-\theta_j|+\sum_{\rho(i)<\rho(j)}(\theta_j-\theta_i)\right)
\end{align*}
$c=\sum_{i,j=1}^p|\theta_i-\theta_j|$ does not depend on the choice of
ranking, so the ranking method that minimises this loss function also
minimises the second term
$\sum_{\rho(i)<\rho(j)}(\theta_j-\theta_i)$. The posterior expected
loss is therefore ${\mathbb
  E}_{\Theta|X}\left(\sum_{\rho(i)<\rho(j)}(\Theta_j-\Theta_i)+c\right)=\sum_{\rho(i)<\rho(j)}({\mathbb
  E}(\Theta_j|X_j)-{\mathbb E}(\Theta_i|X_i))+c$, which is easily seen
to be minimised by posterior mean ranking.

\subsubsection{$r$-value Ranking}

For the $r$-values method, we have the following

\begin{lemma}
  For a given sequence $\theta_1,\ldots,\theta_p$, and a given
  permutation $\rho\in S_p$, let
  $\tau(i)=\left|\{j\in\{1,\ldots,p\}|\theta_j\geqslant\theta_i\}\right|$
  be the true ranking.
Let $R_p(\theta_{\rho^{-1}(1)},\ldots,\theta_{\rho^{-1}(p)})=\sum_{i=1}^n
\left|\rho(i)-\tau(i)\right|$
be the loss function corresponding to the $r$-values method. Let
$L_p(\theta_{\rho^{-1}(1)},\ldots,\theta_{\rho^{-1}(p)})=\sum_{\rho(i)<\rho(j)}
1_{\theta_i<\theta_j}$. $L_p$ is an additive loss function. We have
  $$\frac{1}{2}L_p(\theta_{\rho^{-1}(1)},\ldots,\theta_{\rho^{-1}(p)})\leqslant
R_p(\theta_{\rho^{-1}(1)},\ldots,\theta_{\rho^{-1}(p)})\leqslant 2L_p(\theta_{\rho^{-1}(1)},\ldots,\theta_{\rho^{-1}(p)})$$
\end{lemma}

\begin{proof}
We have that $\tau$ is the permutation such that $\tau(i)<\tau(j)$
whenever $\theta_i>\theta_j$.  The right-hand inequality is easy ---
$\tau\rho^{-1}$ can be generated by a minimal sequence of adjacent
transpositions. The number of these transpositions is
$L_p\left(\theta_{\rho^{-1}}(1),\ldots,\theta_{\rho_p^{-1}}(p)\right)$. The most by which an adjacent transposition can increase the
pointwise loss $R_p$ is 2 --- that is,
$R_p(x_1,\ldots,x_{i+1},x_{i},\ldots x_p)\leqslant R_p(x_1,\ldots,x_{i},x_{i+1},\ldots x_p)+2$. 
  For the second part, we will say that $\tau$ {\em crosses} $\rho$ on
  the pair $(i,j)$ if $\tau$ and $\rho$ disagree whether $i$ should be
  ranked before $j$. Let
  $$R_i=\left\{\begin{array}{ll} \{j|\tau(j)<\tau(i)\leqslant
  \rho(i)<\rho(j)\}\cup\{j|\rho(j)<\rho(i)<\tau(j)\}& \textrm{if
  }\tau(i)\leqslant
  \rho(i)\\ \{j|\tau(j)>\tau(i)>\rho(i)>\rho(j)\}\cup\{j|\tau(j)<\rho(i)<\rho(j)\}&
  \textrm{if }\tau(i)>\rho(i)\end{array}\right.$$
  be the set of
  units $j$ for which $\tau$ crosses $\rho$ on $(i,j)$, but $\tau(j)$
  is not between $\tau(i)$ and $\rho(i)$.  Let
  $$N_i=\left\{\begin{array}{ll}
  \{j|\tau(i)<\tau(j)\leqslant\rho(i)\land\rho(j)<\rho(i)\}&\textrm{if }\tau(i)\leqslant \rho(i)\\
  \{j|\tau(i)>\tau(j)\geqslant\rho(i)\land\rho(j)>\rho(i)\}&\textrm{if }\tau(i)>\rho(i)
  \end{array}\right.$$
  be the set of units $j$ for which $\tau$ crosses $\rho$ on $(i,j)$, and $\tau(j)$ is between
  $\tau(i)$ and $\rho(i)$.
We see that if $\tau$ crosses $\rho$ on $(i,j)$, then we must have $j\in R_{i}\cup
N_i$. Let
$$C_i=\{j|\tau(j)<\tau(i)\land\rho(j)>\rho(i)\}\cup\{j|\tau(j)>\tau(i)\land\rho(j)<\rho(i)\}$$
be the set of units for which $\tau$ crosses $\rho$ on $(i,j)$. We have $|C_i|\leqslant |R_i|+|N_i|$. Furthermore, we have
that $R_p(\theta_{\rho^{-1}(1)},\ldots,\theta_{\rho^{-1}(p)})\geqslant
\sum |R_i|$, since for any $i$ with $j\in R_i$, we have
$\tau(j)\leqslant \rho(i)<\rho(j)$ or $\tau(j)\geqslant \rho(i)>\rho(j)$, so $|\rho(j)-\tau(j)|\geqslant
|\{i|j\in R_i\}|$. Also, $N_i\subseteq
\{j|\tau(i)<\tau(j)\leqslant\rho(i)\}\cup\{j|\tau(i)>\tau(j)\geqslant\rho(i)\}$, so we must have $R_p(\theta_{\rho^{-1}(1)},\ldots,\theta_{\rho^{-1}(p)})\geqslant
\sum_{i=1}^n|N_i|$. This gives us
$$L_p(\theta_{\rho^{-1}(1)},\ldots,\theta_{\rho^{-1}(p)})=\sum_{i=1}^n|C_i|\leqslant
\sum_{i=1}^n|R_i|+\sum_{i=1}^n|N_i|\leqslant 2R_p(\theta_{\rho^{-1}(1)},\ldots,\theta_{\rho^{-1}(p)})$$  
\end{proof}

The main result about restrained pairwise loss functions is the following lemma:

\begin{lemma}\label{PosExpectAdditive}
Let $X_1$ and $X_2$ be independent random variables with means $\mu_1$ and
$\mu_2$, and variances ${\sigma_1}^2$ and ${\sigma_2}^2$
respectively. Let $l$ be a restrained pairwise loss function.
Then
\begin{enumerate}[(i)]
\item if $\mu_1>\mu_2$, ${\mathbb E}(l(X_1,X_2))\leqslant
  D\left(\frac{{\sigma_1}^2+{\sigma_2}^2}{(\mu_1-\mu_2)^2}\land
  1\right)+3\lambda\frac{{\sigma_1}^2+{\sigma_2}^2}{\mu_1-\mu_2}$

\item ${\mathbb E}(l(X_1,X_2))\leqslant D+4\lambda(\mu_2-\mu_1)_++(2+2\sqrt{2})\lambda\sqrt{{\sigma_1}^2+{\sigma_2}^2}$

\item If
  $\mu_1-\mu_2>2\sqrt{\frac{2\lambda({\sigma_1}^2+{\sigma_2}^2)}{b}}$,
  then ${\mathbb
  E}(l(X_2,X_1))\geqslant
\frac{1}{2}\left(\frac{b^3}{(b+\lambda)^2}\right)\left((\mu_1-\mu_2)\land
a\right)$
  \end{enumerate}
\end{lemma}

\begin{proof}
  (i)
We divide the plane into several regions:

\hfil\includegraphics{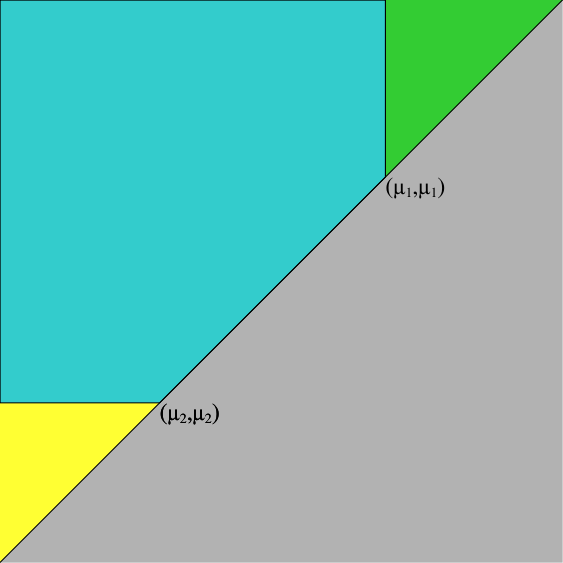}%

In the cyan region, we have
$l(x,y)<l(\mu_2,\mu_1)+\lambda(y-x)$; in the yellow region, we have
$l(x,y)\leqslant l(x,\mu_1)\leqslant l(\mu_2,\mu_1)+\lambda(\mu_2-x)$; in the green region, we
have $l(x,y)\leqslant l(\mu_2,y)\leqslant l(\mu_2,\mu_1)+\lambda(y-\mu_1)$.
It therefore follows that
$${\mathbb E}(l(X_1,X_2))\leqslant
l(\mu_2,\mu_1)P(X_1<X_2)+\lambda\Big({\mathbb E}((X_2-\mu_1)_+)+{\mathbb E}((X_2-X_1)_+)+{\mathbb E}((\mu_2-X_1)_+)\Big)$$
For a random variable
$A$ with finite mean $\mu<0$ and finite variance $\sigma^2$, Chebyshev's inequality gives 
$${\mathbb E}(A_+)=\int_0^\infty P(A>t)\,dt\leqslant
\int_0^\infty \frac{\sigma^2}{(t-\mu)^2}\,dt=\int_{-\mu}^\infty \frac{\sigma^2}{u^2}\,du=-\frac{\sigma^2}{\mu}$$
Thus 
$${\mathbb
  E}((X_2-\mu_1)_+)+{\mathbb E}((X_2-X_1)_+)+{\mathbb
  E}((\mu_2-X_1)_+)\leqslant\frac{{\sigma_2}^2}{(\mu_1-\mu_2)}+\frac{{\sigma_1}^2+{\sigma_2}^2}{(\mu_1-\mu_2)}+\frac{{\sigma_1}^2}{(\mu_1-\mu_2)}=2\frac{{\sigma_1}^2+{\sigma_2}^2}{\mu_1-\mu_2}$$
Chebyshev's inequality also gives us that
$$P(X_1<X_2)=P(X_1-X_2<0)\leqslant \frac{{\sigma_1}^2+{\sigma_2}^2}{(\mu_1-\mu_2)^2}$$
so
$${\mathbb E}(l(X_1,X_2))\leqslant
l(\mu_2,\mu_1)\left(\frac{{\sigma_1}^2+{\sigma_2}^2}{(\mu_1-\mu_2)^2}\land
1\right)+2\lambda\frac{{\sigma_1}^2+{\sigma_2}^2}{\mu_1-\mu_2}$$
Using the fact that $l(\mu_2,\mu_1)<D+\lambda(\mu_1-\mu_2)$, we therefore have
$${\mathbb E}(l(X_1,X_2))\leqslant
D\left(\frac{{\sigma_1}^2+{\sigma_2}^2}{(\mu_1-\mu_2)^2}\land 1\right)+3\lambda\frac{{\sigma_1}^2+{\sigma_2}^2}{\mu_1-\mu_2}$$

(ii)
For a random variable
$A$ with finite mean $\mu$ and finite variance $\sigma^2$,
we have for any $t>0$ with $t>\mu$,
$${\mathbb E}(A_+)\leqslant {\mathbb E}((A-t)_+)+t\leqslant t+\frac{\sigma^2}{t-\mu}$$
It is easy to check that this upper bound is minimised by
$t=\mu+\sigma$, which gives 
$${\mathbb E}(A_+)\leqslant \mu+2\sigma$$

Similarly to (i), if $\mu_2\geqslant \mu_1$, we have
\begin{align*}
{\mathbb E}(l(X_1,X_2))&\leqslant
l(\mu_2,\mu_1)P(X_1<X_2)+\lambda\Big({\mathbb
  E}((X_2-\mu_1)_+)+{\mathbb E}((X_2-X_1)_+)+{\mathbb
  E}((\mu_2-X_1)_+)\Big)\\
&\leqslant
l(\mu_2,\mu_1)+\lambda\Big(3(\mu_2-\mu_1)_++2\left(\sigma_1+\sigma_2+\sqrt{{\sigma_1}^2+{\sigma_2}^2}\right)\Big)\\
&\leqslant D+4\lambda(\mu_2-\mu_1)_++(2+2\sqrt{2})\lambda\sqrt{{\sigma_1}^2+{\sigma_2}^2}
\end{align*}

Since $l(x,y)$ is increasing in $y$, if $\mu_2<\mu_1$, we have for any
$\epsilon>0$, 
$${\mathbb E}(l(X_1,X_2))\leqslant {\mathbb
  E}(l(X_1,X_2+\mu_1-\mu_2+\epsilon))\leqslant
D+4\lambda\epsilon+(2+2\sqrt{2})\lambda\sqrt{{\sigma_1}^2+{\sigma_2}^2}$$
Taking the limit as $\epsilon\rightarrow 0$ completes the proof for $\mu_2<\mu_1$.

(iii) Now we consider the red area in the following diagram:

\hfil\includegraphics{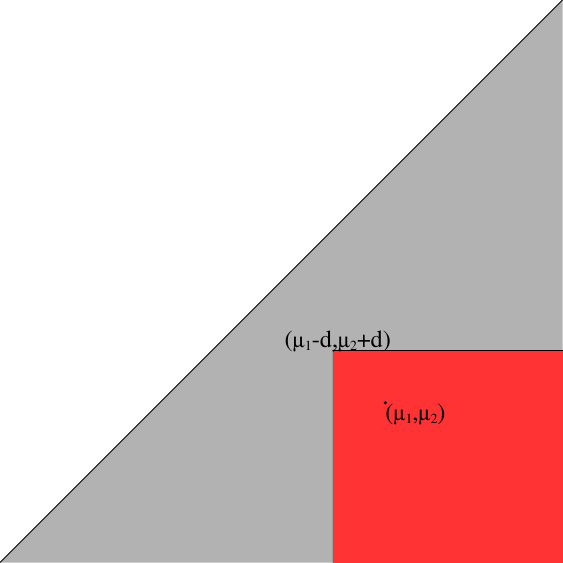}

\noindent where
$d=\sqrt{\frac{b({\sigma_1}^2+{\sigma_2}^2)}{2\lambda}}$. This means
that $l(\mu_2,\mu_1)>b(\mu_1-\mu_2)>2b\sqrt{\frac{2\lambda\left({\sigma_1}^2+{\sigma_2}^2\right)}{b}}=2\sqrt{2\lambda
  b\left({\sigma_1}^2+{\sigma_2}^2\right)}=4\lambda d$.
In the red region, we have $l(y,x)>l(\mu_2+d,\mu_1-d)>l(\mu_2,\mu_1)-2\lambda
d>\frac{l(\mu_2,\mu_1)}{2}$, so
$${\mathbb E}(l(X_2,X_1))\geqslant
\frac{1}{2}l(\mu_2,\mu_1)P(\mu_1-d<X_1)P(\mu_2+d>X_2)$$

The one-sided Chebyshev inequality gives us
\begin{align*}
P(\mu_1-d<X_1)&\geqslant \frac{d^2}{{\sigma_1}^2+d^2}\\
P(\mu_2+d>X_2)&\geqslant \frac{d^2}{{\sigma_2}^2+d^2}\\
P(\mu_1-d<X_1)P(\mu_2+d>X_2)&\geqslant
\left(\frac{d^2}{{\sigma_1}^2+d^2}\right)\left(\frac{d^2}{{\sigma_2}^2+d^2}\right)\\
&=
\frac{\left(\frac{b({\sigma_1}^2+{\sigma_2}^2)}{2\lambda}\right)^2}{\left(\frac{b({\sigma_1}^2+{\sigma_2}^2)}{2\lambda}+{\sigma_1}^2\right)\left(\frac{b({\sigma_1}^2+{\sigma_2}^2)}{2\lambda}+{\sigma_2}^2\right)}\\
&=\frac{b^2}{b^2+2\lambda b+4\lambda^2\frac{{\sigma_1}^2{\sigma_2}^2}{\left({\sigma_1}^2+{\sigma_2}^2\right)^2}}\\
&\geqslant\left(\frac{b}{b+\lambda}\right)^2
\end{align*}
 We have therefore shown that ${\mathbb
  E}(l(X_2,X_1))\geqslant
\frac{1}{2}\left(\frac{b}{b+\lambda}\right)^2l(\mu_2,\mu_1)$.
\end{proof}

Combining this with Lemma~\ref{PostMeanVar} gives

\begin{proposition}\label{PropPairwiseBounds}
  Let $X_1<X_2$ have means $\Theta_1$ and $\Theta_2$ with
  continuous quasiunimodal tail-dominating prior. Let the conditional
  variances be $\sigma_1$ and $\sigma_2$.
  Let $l$ be a restrained binary loss function.
  Then there are
constants $c$, $d$ such that:
\begin{enumerate}[(i)]
\item If $X_2-X_1>2c(|X_1|+|X_j|+2)\left(\sqrt{{\sigma_i}^2+{\sigma_j}^2}\right)$,
  and $c\sqrt{{\sigma_i}^2+{\sigma_j}^2}<1$,
we have that 
$${\mathbb E}(l(\Theta_2,\Theta_1)|X_1,X_2)\leqslant 8c^2D\lambda\left({\sigma_1}^2+{\sigma_2}^2\right)(|X_1|+|X_2|+d)^2\left(\frac{2}{(X_2-X_1)^2}+\frac{3\lambda}{2(X_2-X_1)}\right)$$

\item If $X_2-X_1<2c(|X_1|+|X_j|+2)\left(\sqrt{{\sigma_i}^2+{\sigma_j}^2}\right)$, 
  we have that
$${\mathbb E}(l(\Theta_2,\Theta_1)|X_1,X_2)\leqslant
D+12\lambda c\left(\sqrt{{\sigma_1}^2+{\sigma_2}^2}\right)\left(|X_1|+|X_2|+2+\frac{1}{c}\right)$$

\item If $X_1>X_2+4c(|X_1|+|X_j|+2)\left(\sqrt{{\sigma_i}^2+{\sigma_j}^2}\right)$ we have that 
$${\mathbb E}(l(\Theta_2,\Theta_1)|X_1,X_2)\geqslant
  \frac{1}{2}\left(\frac{b^3}{(b+\lambda)^2}\right)\left(X_1-X_2-c(|X_1|+|X_2|+2)\sqrt{{\sigma_1}^2+{\sigma_2}^2}\right)\land \frac{a}{2}\left(\frac{b^3}{(b+\lambda)^2}\right)$$
  \end{enumerate}

\end{proposition}

\begin{proof}
(i) 
  By Lemma~\ref{PostMeanVar}, the posterior distribution of
$\Theta_2|X_2$ has mean $\mu_2>X_2-c(|X_2|+1)\sigma_2$ and variance
${s_2}^2<c^2{\sigma_2}^2\left(|X_2|+1\right)^2$, while the posterior distribution of
$\Theta_1|X_1$ has mean $\mu_1<X_1+c(|X_1|+1)\sigma_1$ and variance
${s_1}^2<c^2{\sigma_1}^2\left(|X_1|+1\right)^2$. Lemma~\ref{PosExpectAdditive}(i)
 gives us that
$$  {\mathbb E}(l(\Theta_2,\Theta_1)|X_1,X_2)\leqslant
 D\left(\frac{{s_1}^2+{s_2}^2}{(\mu_2-\mu_1)^2}\land
 1\right)+3\lambda\frac{{s_1}^2+{s_2}^2}{\mu_2-\mu_1}$$
Since $X_2-X_1>2c(|X_1|+|X_2|+2)\sqrt{{\sigma_i}^2+{\sigma_j}^2}$, we
have $$c\sigma_2(|X_2|+1)+c\sigma_1(|X_1|+1)<c\sqrt{{\sigma_i}^2+{\sigma_j}^2}\left(|X_1|+|X_2|+2\right)<\frac{X_2-X_1}{2}$$
so $\mu_2-\mu_1>\frac{X_2-X_1}{2}$. We therefore get
$${\mathbb E}(l(\Theta_2,\Theta_1)|X_1,X_2)\leqslant 4c^2D\left({\sigma_1}^2+{\sigma_2}^2\right)(|X_1|+|X_2|+2)^2\left(\frac{1}{(X_2-X_1)^2}+\frac{3\lambda}{2(X_2-X_1)}\right)$$

(ii) In this case, 
Lemma~\ref{PosExpectAdditive}(ii) gives us that
$${\mathbb E}(l(\Theta_1,\Theta_2)|X_1,X_2)\leqslant D+4\lambda(\mu_2-\mu_1)_++(2+2\sqrt{2})\lambda\sqrt{{s_1}^2+{s_2}^2}$$
By Lemma~\ref{PostMeanVar}, we have that 
$$(\mu_2-\mu_1)_+\leqslant 3c\sqrt{{\sigma_1}^2+{\sigma_2}^2}(|X_1|+|X_2|+2)$$
so
$${\mathbb E}(l(\Theta_1,\Theta_2)|X_1,X_2)\leqslant D+12\lambda
c\sqrt{{\sigma_1}^2+{\sigma_2}^2}(|X_1|+|X_2|+2)+(2+2\sqrt{2})\lambda\sqrt{{\sigma_1}^2+{\sigma_2}^2}\leqslant
D+12\lambda c\sqrt{{\sigma_1}^2+{\sigma_2}^2}\left(|X_1|+|X_2|+2+\frac{1}{c}\right)$$

(iii) As in (i), we have that 
  by Lemma~\ref{PostMeanVar}, the posterior distribution of
$\Theta_2|X_2$ has mean $\mu_2>X_2-c(|X_2|+1)\sigma_2$ and variance
${s_2}^2<c^2{\sigma_2}^2\left(|X_2|+1\right)^2$, while the posterior distribution of
$\Theta_1|X_1$ has mean $\mu_1>X_1-c(|X_1|+1)\sigma_1$ and variance
${s_1}^2<c^2{\sigma_1}^2\left(|X_1|+1\right)^2$. Lemma~\ref{PosExpectAdditive}(iii)
  then gives us that
  \begin{align*}
  {\mathbb E}(l(\Theta_1,\Theta_2)|X_1,X_2)&\geqslant
  \frac{1}{2}\left(\frac{b^3}{(b+\lambda)^2}\right)\left((X_1-c(|X_1|+1)\sigma_1-(X_2+c(|X_2|+1)\sigma_2))\land
  a\right)\\&
  \geqslant \frac{1}{2}\left(\frac{b^3}{(b+\lambda)^2}\right)\left((X_1-X_2-c(|X_1|+|X_2|+2)\sqrt{{\sigma_1}^2+{\sigma_2}^2})\land a\right)
\end{align*}
  
\end{proof}

\section{Main Theorem}\label{Proof}

\subsection{Consistency Framework}

Before we can state our main theorem, we need to clarify what
consistency means in the context of the ranking problem.
We define consistency in terms of a loss function
$L_p$. We will say that the ranking method which estimates the
permutation $\hat{\rho}_p$ at stage $p$ is consistent if ${\mathbb
  E}(L_p({\hat{\rho}_p}^{-1}(1),{\hat{\rho}_p}^{-1}(2),\ldots,{\hat{\rho}_p}^{-1}(p))\rightarrow 0$ as
$p\rightarrow 0$. We will allow $L_p$ to be different
from the loss function used by the ranking method. This allows a
certain amount of misspecification. In the case where $L_p$ is an
additive loss function, this is where the scale function $s(p)$ is
important. There are three natural choices of scale function,
depending on the strength of convergence we need: total misranking
loss, where $s(p)=1$; per-unit misranking loss, where $s(p)=p^{-1}$;
and per-pair misranking loss, where $s(p)=p^{-2}$. Most loss funcions
in the literature are described in the total misranking loss form, but
because those papers only describe the method, rather than considering
the consistency, the choice of form to describe is arbitrary. For a
regular loss function with a discontinuity at $\theta_i=\theta_j$: a method is
total misranking loss convergent if asymptotically, the probability of
the correct ranking converges to 1; it is per-unit consistent if for
almost all units, the probability of them being misranked relative to
a randomly chosen unit converges to 0; it is per-pair consistent if
the probability of a pair of randomly chosen units being misranked
converges to 0.

\begin{lemma}\label{ConsistencyConditions}
A ranking method $\hat{\rho}_p$ is consistent with respect to a regular loss
function $L_p$ with scaling factor $s(p)$ if the following
two conditions hold for the set of misranked pairs ${\mathcal M}_p=\{(i,j)|\Theta_i<\Theta_j,\hat{\rho}_p(i)<\hat{\rho}_p(j)\}$:

\begin{itemize}

  \item
    $s(p){\mathbb E}\left(|{\mathcal M}_p|\right)\rightarrow 0$

\item $s(p)\sum_{i,j=1}^p(\Theta_j-\Theta_i)P\left((i,j)\in {\mathcal
  M}_p\right)\rightarrow 0$
    
\end{itemize}
\end{lemma}

\begin{proof}
Suppose the conditions hold. The expected loss is ${\mathbb
  E}\left(L_p\left(\Theta_{{\hat{\rho}_p}^{-1}(1)},\ldots,\Theta_{{\hat{\rho}_p}^{-1}(p)}\right)\right)$. Since
$L_p$ is regular, it is equivalent to an additive loss
generated by a Lipschitz pairwise loss function $l(x,y)$.  By
the equivalence, we have that
$
L_p\left(\Theta_{{\hat{\rho}_p}^{-1}(1)},\ldots,\Theta_{{\hat{\rho}_p}^{-1}(p)}\right)\leqslant
us(p)\sum_{\hat{\rho}_p(i)<\hat{\rho}_p(j)}l(\Theta_{i},\Theta_{j})$. By definition, if
$\hat{\rho}_p(i)<\hat{\rho}_p(j)$, but $(i,j)\not\in{\mathcal M}$, we have
$l(\Theta_i,\Theta_j)=0$, so we have 
$\sum_{\hat{\rho}_p(i)<\hat{\rho}_p(j)}l(\Theta_i,\Theta_j)=\sum_{(i,j)\in{\mathcal
    M}}l(\Theta_i,\Theta_j)$. Therefore, it is sufficient to prove
that $l(\Theta_i,\Theta_j)<a(\Theta_j-\Theta_i)+b$ for some constants
$a$ and $b$. This follows from the definition of a restrained pairwise
loss, with $a=\lambda$ and $b=D$.
\end{proof}

\subsection{Main Theorem}

We will prove consistency by showing that asymptotically, all standard
ranking methods agree with value ranking for the problems we are
considering. Consistency will then follow from the consistency of
value ranking.

\begin{lemma}\label{ValRankConsist}
If the $\Theta_i$ are drawn i.i.d. from a fixed continuous
distribution with bounded density and finite mean and variance, and
$ps(p)\sum_{i=1}^p \sigma_{ip}\rightarrow 0$, then value ranking is
consistent.
\end{lemma}

\begin{proof}
The loss function is
$s(p)\sum_{i,j=1}^pl(\Theta_i,\Theta_j)P(X_{ip}<X_{jp})$. By
Chebyshev's inequality, we have
$$P(X_{ip}<X_{jp})<\frac{({\sigma_{ip}}^2+{\sigma_{jp}}^2)}{(\Theta_{j}-\Theta_{i})^2}\land 1$$
Since $l(\Theta_i,\Theta_j)<D+\lambda(\Theta_j-\Theta_i)_+$, the loss
function is
\begin{align*}
s(p)\sum_{i,j=1}^pl(\Theta_i,\Theta_j)P(X_{ip}<X_{jp})&\leqslant
\sum_{\Theta_i<\Theta_j}(D+\lambda(\Theta_j-\Theta_i))\left(\frac{({\sigma_{ip}}^2+{\sigma_{jp}}^2)}{(\Theta_{i}-\Theta_{j})^2}\land
1\right)\\
&=\sum_{\Theta_i<\Theta_j}(D+\lambda(\Theta_j-\Theta_i))\left({\sigma_{ip}}^2+{\sigma_{jp}}^2\right)\left({(\Theta_{i}-\Theta_{j})^{-2}}\land
  \left({\sigma_{ip}}^2+{\sigma_{jp}}^2\right)^{-1}\right)
\end{align*}
Since $\Theta_j$ follows a continuous distribution with density
$\pi(\theta)<C$, for some upper bound $C$, we have that for fixed $i$,
$P\left(|\Theta_j-\Theta_i|^{-2}>t\right)=P\left(|\Theta_j-\Theta_i|<t^{-\frac{1}{2}}\right)\leqslant
t^{-\frac{1}{2}}C$. This means that for any $j$, $${\mathbb
  E}\left((\Theta_{i}-\Theta_{j})^{-2}\land t\right)=\int_0^t
P\left((\Theta_{i}-\Theta_{j})^{-2}>u\right)\,du\leqslant C\int_0^t
u^{-\frac{1}{2}}\,du=2C\sqrt{t}$$
Thus the expected loss is 
$${\mathbb
  E}\left(s(p)\sum_{i,j=1}^pl(\Theta_i,\Theta_j)P(X_{ip}<X_{jp})\right)\leqslant
s(p){\mathbb
  E}\left(\sum_{\Theta_i<\Theta_j}(D+\lambda(\Theta_j-\Theta_i))\sqrt{{\sigma_{ip}}^2+{\sigma_{jp}}^2}\right)\leqslant
ps(p)\sum_{i=1}^p\sigma_{ip}$$
Thus if $ps(p)\sum_{i=1}^p\sigma_{ip}\rightarrow 0$ then value ranking
is consistent.
\end{proof}

\begin{lemma}\label{Inequality}
  For any 4 real numbers $x,y,z,w$, we have
  $$\left|\frac{x-y}{|x|+|y|+2}\right|-\left|\frac{z-w}{|z|+|w|+2}\right|\leqslant\frac{1}{2}\left(|x-z|+|y-w|\right)\left(1+\frac{|x-z|+|y-w|}{2}\right)$$
\end{lemma}

\begin{proof}
  \begin{align*}
  \left|\frac{|x|+|y|+2}{|z|+|w|+2}-1\right|&=
  \left|\frac{|x|+|y|+2-\left(|z|+|w|+2\right)}{|z|+|w|+2}\right|\\
  &=\left|\frac{|x|-|z|+|y|-|w|}{|z|+|w|+2}\right|\\
  &\leqslant \frac{|x-z|+|y-w|}{2}\\
\left|x-y\right|-  \frac{\left(|x|+|y|+2\right)}{\left(|z|+|w|+2\right)}\left|
z-w\right|&\leqslant|x-y|-|z-w|+|z-w|-
\frac{\left(|x|+|y|+2\right)}{\left(|z|+|w|+2\right)}\left|z-w\right|\\
&\leqslant|x-z|+|y-w|+|z-w|\left|1-  \frac{\left(|x|+|y|+2\right)}{\left(|z|+|w|+2\right)}\right|\\
&\leqslant\left(|x-z|+|y-w|\right)\left(1+\frac{|z-w|}{2}\right)\\
  \left|
\frac{x-y}{|x|+|y|+2}\right|-\left|
\frac{z-w}{|z|+|w|+2}\right|&\leqslant\left(|x-z|+|y-w|\right)\left(\frac{1+\frac{|z-w|}{2}}{|x|+|y|+2}\right)\\
&\leqslant\frac{1}{2}\left(|x-z|+|y-w|\right)\left(\frac{|z|+|w|+2}{|x|+|y|+2}\right)\\
&\leqslant\frac{1}{2}\left(|x-z|+|y-w|\right)\left(1+\frac{|x-z|+|y-w|}{2}\right)
\end{align*}
\end{proof}

\begin{lemma}\label{ValRankPostLoss}
For a standard ranking problem, where the true values follow a
continuous distribution with bounded density, using a regular ranking
method, the expected posterior total loss of value ranking is
$$V_p={\mathbb
    E}_{\Theta|X,\pi}\left(R_p\left(\Theta_{\tau_p^{-1}(1)},\ldots,\Theta_{\tau_p^{-1}(p)}\right)\right)=O\left(\sum_{i,j=1}^p \left({\sigma_{ip}}^2+{\sigma_{jp}}^2\right)^{\frac{1}{3}}\right)$$
\end{lemma}

\begin{proof}
  Let $\tau_p$ be the value ranking of the data. That is,
  $\tau_p(i)<\tau_p(j)$ if $X_{ip}>X_{jp}$. We define a set ${\mathcal
    C}_p$ of all close pairs by $(i,j)\in{\mathcal C}_p$
  if $$0\leqslant
  X_{ip}-X_{jp}<4c(|X_{ip}|+|X_{jp}|+2)\left(\sqrt{{\sigma_{ip}}^2+{\sigma_{jp}}^2}\right)^{\frac{2}{3}}$$
  From parts (i) and (ii) of Proposition~\ref{PropPairwiseBounds}, we
  get
\begin{align*}
V_p&\leqslant
u\sum_{(i,j)\in {\mathcal
    C}_p}D+4\lambda c
\left(\left(\sqrt{{\sigma_{ip}}^2+{\sigma_{jp}}^2}\right)\left(|X_{ip}|+|X_{jp}|+2+\frac{1}{c}\right)\right)\\
&\qquad+u\sum_{\stackrel{\tau(i)<\tau(j)}{(i,j)\not\in{\mathcal C}_p}}8Dc^2\lambda\left({\sigma_{ip}}^2+{\sigma_{jp}}^2\right)(|X_{ip}|+|X_{jp}|+d)^2\left(\frac{2}{(X_{jp}-X_{ip})^2}+\frac{3\lambda}{2(X_{jp}-X_{ip})}\right)
\end{align*}
For $(i,j)\not\in{\mathcal C}_p$, we have
$$8Dc^2\lambda\left({\sigma_{ip}}^2+{\sigma_{jp}}^2\right)(|X_{ip}|+|X_{jp}|+d)^2\left(\frac{2}{(X_{jp}-X_{ip})^2}+\frac{3\lambda}{2(X_{jp}-X_{ip})}\right)=O\left(\left({\sigma_{ip}}^2+{\sigma_{jp}}^2\right)^{\frac{1}{3}}\right)$$

On the other hand, for $(i,j)\in{\mathcal C}_p$, we have $$
D+4\lambda c
\left(\left(\sqrt{{\sigma_{ip}}^2+{\sigma_{jp}}^2}\right)\left(|X_{ip}|+|X_{jp}|+2+\frac{1}{c}\right)\right)=O(1)$$
Thus
$$V_p=O(|{\mathcal C}_p|)+\sum_{i,j=1}^p
O\left(\left({\sigma_{ip}}^2+{\sigma_{jp}}^2\right)^{\frac{1}{3}}\right)$$
We have that
$${\mathbb E}\left(|{\mathcal C}_p|\right)=\sum_{i,j=1}^p
P\left((i,j)\in {\mathcal C}_p\right)=\sum_{i,j=1}^p
P\left(0\leqslant
\frac{X_{ip}-X_{jp}}{|X_{ip}|+|X_{jp}|+2}<4c\left({{\sigma_{ip}}^2+{\sigma_{jp}}^2}\right)^{\frac{1}{3}}\right)$$
Thus is is sufficient to show that $$P\left(0\leqslant
\frac{X_{ip}-X_{jp}}{|X_{ip}|+|X_{jp}|+2}<4c\left({{\sigma_{ip}}^2+{\sigma_{jp}}^2}\right)^{\frac{1}{3}}\right)=O\left(\left({\sigma_{ip}}^2+{\sigma_{ip}}^2\right)^{\frac{1}{3}}\right)$$
We can show for any $A>0$,
\begin{align*}
P\left(0\leqslant
\frac{X_{ip}-X_{jp}}{|X_{ip}|+|X_{jp}|+2}<4c\left({{\sigma_{ip}}^2+{\sigma_{jp}}^2}\right)^{\frac{1}{3}}\right)&\leqslant
P\left(-4c\left({{\sigma_{ip}}^2+{\sigma_{jp}}^2}\right)^{\frac{1}{3}}\leqslant
\frac{X_{ip}-X_{jp}}{|X_{ip}|+|X_{jp}|+2}<4c\left({{\sigma_{ip}}^2+{\sigma_{jp}}^2}\right)^{\frac{1}{3}}\right)\\
&\leqslant
P\left(\left|
\frac{\Theta_{i}-\Theta_{j}}{|\Theta_{i}|+|\Theta_{j}|+2}\right|<(4+A)c\left({{\sigma_{ip}}^2+{\sigma_{jp}}^2}\right)^{\frac{1}{3}}\right)\\
&\qquad\qquad+P\left(\left|
\frac{X_{ip}-X_{jp}}{|X_{ip}|+|X_{jp}|+2}\right|-\left|
\frac{\Theta_{i}-\Theta_{j}}{|\Theta_{i}|+|\Theta_{j}|+2}\right|>Ac\left({{\sigma_{ip}}^2+{\sigma_{jp}}^2}\right)^{\frac{1}{3}}\right)\\
\end{align*}
Lemma~\ref{Inequality} tells us that 
$$\left|
\frac{X_{ip}-X_{jp}}{|X_{ip}|+|X_{jp}|+2}\right|-\left|
\frac{\Theta_{i}-\Theta_{j}}{|\Theta_{i}|+|\Theta_{j}|+2}\right|\leqslant \frac{\left(|X_{ip}-\Theta_i|+|X_{jp}-\Theta_j|\right)}{2}\left(1+\frac{|X_{ip}-\Theta_i|+|X_{jp}-\Theta_j|}{2}\right)$$
Letting
$\Delta=\frac{\left(|X_{ip}-\Theta_i|+|X_{jp}-\Theta_j|\right)}{2}$,
we have $$
P\left(\left|
\frac{X_{ip}-X_{jp}}{|X_{ip}|+|X_{jp}|+2}\right|-\left|
\frac{\Theta_{i}-\Theta_{j}}{|\Theta_{i}|+|\Theta_{j}|+2}\right|>Ac\left({{\sigma_{ip}}^2+{\sigma_{jp}}^2}\right)^{\frac{1}{3}}\right)\leqslant
P\left(\Delta+\Delta^2>Ac\left({{\sigma_{ip}}^2+{\sigma_{jp}}^2}\right)^{\frac{1}{3}}\right)$$
For small enough $\sigma$, we have 
$Ac\left({{\sigma_{ip}}^2+{\sigma_{jp}}^2}\right)^{\frac{1}{3}}<1$, so 
$$P\left(\Delta+\Delta^2>Ac\left({{\sigma_{ip}}^2+{\sigma_{jp}}^2}\right)^{\frac{1}{3}}\right)<P\left(2\Delta>Ac\left({{\sigma_{ip}}^2+{\sigma_{jp}}^2}\right)^{\frac{1}{3}}\right)\leqslant P\left(|X_{ip}-\Theta_i|>Ac{\sigma_{ip}}^{\frac{2}{3}}\right)+P(\left(|X_{jp}-\Theta_{jp}|>AC{\sigma_{jp}}^{\frac{2}{3}}\right)$$
Since $X_{ip}-\Theta_i$ has mean 0 and variance ${\sigma_{ip}}^2$,
Chebyshev's inequality
gives $$P\left(|X_{ip}-\Theta_i|>Ac{\sigma_{ip}}^{\frac{2}{3}}\right)\leqslant \frac{{\sigma_{ip}}^2}{\left(Ac{\sigma_{ip}}^{\frac{2}{3}}\right)^2}=\frac{{\sigma_{ip}}^{\frac{2}{3}}}{A^2c^2}$$
Meanwhile, since $\Theta_i$ and $\Theta_j$ are i.i.d. continuous
random variables with bounded density, and independent
of $\sigma_{ip}$ and $\sigma_{jp}$, the value $\left|
\frac{\Theta_{i}-\Theta_{j}}{|\Theta_{i}|+|\Theta_{j}|+2}\right|$ has
a continuous distribution with bounded density on a neighbourhood of
0. Let the density be bounded by $\nu$. Then we have 
$$P\left(\left|
\frac{\Theta_{i}-\Theta_{j}}{|\Theta_{i}|+|\Theta_{j}|+2}\right|<(4+A)c\left({{\sigma_{ip}}^2+{\sigma_{jp}}^2}\right)^{\frac{1}{3}}\right)\leqslant (4+A)c\left({{\sigma_{ip}}^2+{\sigma_{jp}}^2}\right)^{\frac{1}{3}}\nu$$
\end{proof}

The main result of this paper is the following theorem:

\begin{theorem}\label{MainTheorem}  
  If the $\Theta_i$ are all drawn i.i.d. from a fixed continuous
  distribution with finite mean and variance, and $\sigma_{ip}$ are
  i.i.d. with finite mean, then any ranking method is consistent (with
  respect to a regular loss function) provided the following
  conditions hold.

  \begin{itemize}

  \item $s(p)p^{2}{\mathbb E}({\sigma_{ip}})^{\frac{1}{3}}\rightarrow
    0$.

  \item The loss function $R_p$ used for the ranking method is
    equivalent to an additive loss function generated by a restrained
    pairwise loss.

  \item The distribution of $\Theta_i$ has bounded mean and variance.

  \item The estimating prior distribution $\pi$ is tail-dominating.

  \end{itemize}
    
\end{theorem}

\begin{proof}
  Let
  $\hat{\rho}_p$ denote the estimated ranking. By
  definition,
$${\mathbb
  E}_{\Theta|X,\pi}\left(R_p\left(\Theta_{\hat{\rho}_p^{-1}(1)},\ldots,\Theta_{\hat{\rho}_p^{-1}(p)}\right)\right)\leqslant
  V_p=O\left(\left({\sigma_{ip}}^2+{\sigma_{jp}}^2\right)^{\frac{2}{3}}\right)$$
On the other hand,
$${\mathbb
  E}_{\Theta|X,\sigma,\pi}R_p\left(\Theta_{\hat{\rho}_p^{-1}(1)},\ldots,\Theta_{\hat{\rho}_p^{-1}(p)}\right)\geqslant
\frac{c}{2}\left(\frac{b^3}{(b+\lambda)^2}\right)\sum_{(i,j)\in
  {\mathcal D}_p}
\left(X_{ip}-X_{jp}-c(|X_{ip}|+|X_{jp}|+2)\sqrt{{\sigma_{ip}}^2+{\sigma_{jp}}^2}\right)\land a$$
where ${\mathcal D}_p$is the set of pairs on which $\hat{\rho}_p$ and
$\tau_p$ disagree, defined by $${\mathcal D}_p=\{(i,j)|(X_{ip}>X_{jp})\land(\hat{\rho}_p(i)>\hat{\rho}_p(j))\}$$
Thus, we have shown that
\begin{align*}
\sum_{(i,j)\in
  {\mathcal D}_p}
\left(X_{ip}-X_{jp}-c(|X_{ip}|+|X_{jp}|+2)\sqrt{{\sigma_{ip}}^2+{\sigma_{jp}}^2}\right)\land a&=O\left(\sum_{i,j=1}^p\left({\sigma_{ip}}^2+{\sigma_{jp}}^2\right)^{\frac{1}{3}}\right)
\end{align*}
Since
$c(|X_{ip}|+|X_{jp}|+2)\sqrt{{\sigma_{ip}}^2+{\sigma_{jp}}^2}=o\left(\sum_{i,j=1}^p\left({\sigma_{ip}}^2+{\sigma_{jp}}^2\right)^{\frac{1}{3}}\right)$,
we deduce 
\begin{align*}
\sum_{(i,j)\in
  {\mathcal D}_p}
\left(X_{ip}-X_{jp}\right)\land a&=O\left(\sum_{i,j=1}^p\left({\sigma_{ip}}^2+{\sigma_{jp}}^2\right)^{\frac{1}{3}}\right)
\end{align*}
We want to show that $|{\mathcal D}_p|=O\left(\sum_{i,j=1}^p \left({\sigma_{ip}}^2+{\sigma_{jp}}^2\right)^{\frac{1}{6}}\right)$
For any $K>0$, we have that the number of pairs $(i,j)$ for which
$0<X_{ip}-X_{jp}<K\left({\sigma_{ip}}^2+{\sigma_{jp}}^2\right)^{\frac{1}{6}}$
is
$O\left(p^2\left({\sigma_{ip}}^2+{\sigma_{jp}}^2\right)^{\frac{1}{6}}\right)$
Thus, at least $|{\mathcal
  D}_p|-O\left(p^2\left({\sigma_{ip}}^2+{\sigma_{jp}}^2\right)^{\frac{1}{6}}\right)$
of the pairs $(i,j)\in {\mathcal D}_p$ satisfy $X_{ip}-X_{jp}\geqslant
K\left({\sigma_{ip}}^2+{\sigma_{jp}}^2\right)^{\frac{1}{6}}$. Thus
\begin{align*}
 \left(|{\mathcal D}_p|-O\left(p^2\left({\sigma_{ip}}^2+{\sigma_{jp}}^2\right)^{\frac{1}{6}}\right)\right)K\left({\sigma_{ip}}^2+{\sigma_{jp}}^2\right)^{\frac{1}{6}}&\leqslant \sum_{(i,j)\in
  {\mathcal D}_p}\left(X_{ip}-X_{jp}\right)\land
 a=O\left(\sum_{i,j=1}^p\left({\sigma_{ip}}^2+{\sigma_{jp}}^2\right)^{\frac{1}{3}}\right)\\
 \left(|{\mathcal  D}_p|-O\left(p^2\left({\sigma_{ip}}^2+{\sigma_{jp}}^2\right)^{\frac{1}{6}}\right)\right)&=O\left(\sum_{i,j=1}^p\left({\sigma_{ip}}^2+{\sigma_{jp}}^2\right)^{\frac{1}{6}}\right)\\
|{\mathcal  D}_p|&=O\left(\sum_{i,j=1}^p\left({\sigma_{ip}}^2+{\sigma_{jp}}^2\right)^{\frac{1}{6}}\right)\\
\end{align*}
Thus if $s(p)p^2{\mathbb
  E}\left({\sigma_{ip}}^{\frac{1}{3}}\right)\rightarrow 0$, then
$s(p)|{\mathcal D}_p|\rightarrow 0$.

We know that the set of misranked pairs ${\mathcal M}_p$ from
Lemma~\ref{ConsistencyConditions} is contained in the
union $${\mathcal D}_p\cup\{(i,j)|\Theta_i<\Theta_j,X_{ip}>X_{jp}\}$$
Thus, since value ranking is consistent by Lemma~\ref{ValRankConsist},
it is sufficient to prove 
$$s(p)|{\mathcal D}_p|\rightarrow 0\qquad\textrm{and}\qquad s(p)\sum_{(i,j)\in {\mathcal D}_p}(\Theta_j-\Theta_i)_+\rightarrow
0$$
\end{proof}

\section{Examples}\label{Examples}

The key condition in Theorem~\ref{MainTheorem} is that the prior
distribution should be at least as heavy-tailed as the error
distribution. In this section, we provide examples where value ranking
is consistent, but posterior mean ranking is not consistent because
the error distribution is heavier-tailed than the prior
distribution. 

We use posterior mean ranking because it is relatively easy to analyse
--- it is sufficient to show that the pairwise ranking of a pair of
units is incorrect, which can be done by bounding the posterior
means. Furthermore, for these examples the conditional variances for
each unit follow a distribution with a point mass at 0. This is
convenient for proving the results, because for units with $\sigma=0$,
we know that the $\mu_i=x_i=\theta_i$, so we do not need to worry
about finding lower bounds for the posterior mean. It is natural to
assume that similar inconsistency results will also hold for other
ranking methods, and for continuous distributions for
$\sigma_{ip}$. However, the proofs in these cases would be more challenging.

In our first example, the error distribution is heavy-tailed, and the
prior distribution is normal. In the second, the error distribution is
normal, and the prior has very light tails.

\subsection{General Error}

Suppose the true distribution of $\Theta$ is a Pareto distribution
with $\theta_{\textrm{min}}=1$ and $\alpha=4$. Suppose we model the
data using a normal prior with mean and variance estimated from the
data. Asymptotically, these estimates will converge to the true mean
and variance, which are $1.25$ and $\frac{2}{9}$. Suppose the error
distribution $X_{ip}-\Theta_{i}$ has density function
$$f(x)=\frac{\sqrt{2}}{\pi\sigma_{ip}\left(1+\left(\frac{x}{\sigma_{ip}}\right)^4\right)}$$
This has mean 0 and variance ${\sigma_{ip}}^2$. Suppose that
$\sigma_{ip}$ is zero with probability $\frac{1}{2}$, and otherwise
follows an exponential distribution with mean $v_p$, where $p^\alpha
v_p\rightarrow a>0$ for some $a$ and $\alpha$.

By Lemma~\ref{ValRankConsist}, we know that value ranking is total loss
consistent for this problem provided $p^2v_p\rightarrow 0$. We will
show that there are sequences $v_p$ with this property for which
posterior mean ranking is not consistent.

\begin{lemma}\label{PMBound}
  For a random parameter $\Theta$ with normal prior with mean $\mu$
  and variance $\tau^2$, suppose we have an observation $X=\Theta+E$ where $E$ has
  density
  function $$f_E(x)=\frac{\sqrt{2}}{\pi\sigma\left(1+\left(\frac{x}{\sigma}\right)^4\right)}$$
  Suppose that $x-\mu>\frac{2\sqrt[4]{27}\tau^2}{\sigma}$, $x-\mu>2\sigma$ and $x-\mu>8\tau$.
  Then the posterior mean of $\Theta$ satisfies
  $${\mathbb E}(\Theta|X=x)\leqslant \frac{x+\mu}{2}+\frac{289}{4096}e^{2}\tau^2\sigma^{-4}(x-\mu)^{5}e^{-\frac{(x-\mu)^2}{8\tau^2}}$$
\end{lemma}

\begin{proof}
The posterior density is proportional to
$$\pi(\theta)f(x-\theta)\propto\frac{e^{-\frac{(\theta-\mu)^2}{2\tau^2}}}{\sigma^4+\left(x-\theta\right)^4}$$
Now for $\mu<\theta<\mu+\frac{16\tau^2}{x-\mu}<\mu+2\tau$, we have
$\sigma^4+(x-\theta)^4<(x-\mu)^4+\sigma^4<\frac{17}{16}(x-\mu)^4$, so that
$\frac{e^{-\frac{(\theta-\mu)^2}{2\tau^2}}}{\sigma^4+\left(x-\theta\right)^4}>\frac{16}{17}e^{-2}(x-\mu)^{-4}$.
Therefore $$\int_{-\infty}^{\infty}\pi(\theta)f(x-\theta)\,d\theta>\int_{\mu}^{\mu+\frac{16\tau^2}{x-\mu}}\pi(\theta)f(x-\theta)\,d\theta>\frac{16\tau^2}{x-\mu}\times \frac{16}{17}e^{-2}(x-\mu)^{-4}=\frac{256}{17}e^{-2}(x-\mu)^{-5}$$
while for $\theta>\frac{x+\mu}{2}$, we have 
$$\pi(\theta)f(x-\theta)\propto
\frac{e^{-\frac{(\theta-\mu)}{2\tau^2}}}{\sigma^4+\left(x-\theta\right)^4}<\frac{e^{-\frac{(\theta-\mu)^2}{2\tau^2}}}{\sigma^4}=\frac{e^{-\frac{\left(\theta-\mu-\frac{x-\mu}{2}+\frac{x-\mu}{2}\right)^2}{2\tau^2}}}{\sigma^4}<e^{-\frac{(x-\mu)^2}{8\tau^2}}\sigma^{-4}e^{-\frac{(x-\mu)\left(\theta-\mu-\frac{x-\mu}{2}\right)}{2\tau^2}}$$
This gives us
\begin{align*}
\int_{\frac{x+\mu}{2}}^\infty
(\theta-\mu)\pi(\theta)f(x-\theta)\,d\theta&<e^{-\frac{(x-\mu)^2}{8\tau^2}}\sigma^{-4}\int_{\frac{x+\mu}{2}}^\infty
(\theta-\mu)
e^{-\frac{(x-\mu)\left(\theta-\frac{x+\mu}{2}\right)}{2\tau^2}}\,d\theta\\
&=e^{-\frac{(x-\mu)^2}{8\tau^2}}\sigma^{-4}\left(\frac{2\tau^2}{(x-\mu)}\left[-(\theta-\mu)
  e^{-\frac{(x-\mu)\left(\theta-\frac{(x+\mu)}{2}\right)}{2\tau^2}}\right]_{\frac{(x+\mu)}{2}}^\infty+\frac{2\tau^2}{(x-\mu)}\int_{\frac{(x+\mu)}{2}}^\infty
e^{-\frac{(x-\mu)\left(\theta-\frac{(x+\mu)}{2}\right)}{2\tau^2}}\,d\theta\right)\\
&=e^{-\frac{(x-\mu)^2}{8\tau^2}}\sigma^{-4}\left(\tau^2+\frac{4\tau^4}{(x-\mu)^2}\right)\\
&<\frac{17}{16}\tau^2\sigma^{-4}e^{-\frac{(x-\mu)^2}{8\tau^2}}
\end{align*}
Therefore
\begin{align*}
{\mathbb E}(\Theta|X=x)&=\mu+\frac{\int_{-\infty}^\infty
  (\theta-\mu)\pi(\theta)f(x-\theta)\,d\theta}{\int_{-\infty}^{\infty}\pi(\theta)f(x-\theta)\,d\theta}\\
&=\mu+\frac{\int_{-\infty}^{\frac{(x+\mu)}{2}}
  (\theta-\mu)\pi(\theta)f(x-\theta)\,d\theta}{\int_{-\infty}^{\infty}\pi(\theta)f(x-\theta)\,d\theta}+\frac{\int_{\frac{(x+\mu)}{2}}^\infty
  (\theta-\mu)\pi(\theta)f(x-\theta)\,d\theta}{\int_{-\infty}^{\infty}\pi(\theta)f(x-\theta)\,d\theta}\\
&<\mu+{\mathbb E}\left(\Theta-\mu\middle|X=x,\Theta<\frac{x+\mu}{2}\right)+\frac{17}{256}e^{2}(x-\mu)^{5}\times
\frac{17}{16}\tau^2\sigma^{-4}e^{-\frac{(x-\mu)^2}{8\tau^2}}\\
&<\frac{x+\mu}{2}+\frac{17}{256}e^{2}(x-\mu)^{5}\times
\frac{17}{16}\tau^2\sigma^{-4}e^{-\frac{(x-\mu)^2}{8\tau^2}}\\
\end{align*}
\end{proof}

\begin{theorem}
For the following ranking problem:
\begin{itemize}
  
\item The true distribution of $\Theta$ is a Pareto distribution
  with $\theta_{\textrm{min}}=1$ and $\alpha=4$.


\item The error distribution has density function
$$f(x)=\frac{\sqrt{2}}{\pi\sigma_{ip}\left(1+\left(\frac{x}{\sigma_{ip}}\right)^4\right)}$$

\item $\sigma_{ip}$ is zero with probability $\frac{1}{2}$, and otherwise
follows an exponential distribution with mean $v_p$, where $p^\alpha
v_p\rightarrow a>0$ for some $a$ and $\alpha$.
\end{itemize}

\noindent posterior mean ranking with normal prior with mean and variance
estimated is inconsistent with respect to total per-unit misranking loss.
\end{theorem}

\begin{proof}
  Recall that the ranking method with estimated ranking $\hat\rho$ is consistent with respect to total
  per-unit misranking loss if and only if
  $$p^{-1}\sum_{\theta_i<\theta_j} P(\hat{\rho}(i)>\hat{\rho}(j))\rightarrow 0$$
We therefore want to show that this sequence does not converge to 0. We
will do this by constructing a sequence $\zeta_1,\zeta_2,\ldots$ and a
constant $m$ such that

\begin{enumerate}[(i)]

\item \label{GeneralCounterExampleMisrankCond} Whenever
  $\theta_i\in\left[\mu+\frac{3}{4}\zeta_p,\mu+\frac{5}{6}\zeta_p\right]$
  and $\theta_j\in\left[\mu+\frac{11}{12}\zeta_p,\mu+\zeta_p\right]$,
  the probability that $\hat{\rho}(i)<\hat{\rho}(j)$ is bounded below by $m$.

\item \label{GeneralCounterExampleProbCond}
  $P\left(\theta_i\in\left[\mu+\frac{3}{4}\zeta_p,\mu+\frac{5}{6}\zeta_p\right]\right)$
  and
  $P\left(\theta_i\in\left[\mu+\frac{11}{12}\zeta_p,\mu+\zeta_p\right]\right)$
  are both larger than $p^{-\frac{1}{2}}$
  for all sufficiently large $p$.

\end{enumerate}

\noindent These conditions will give us that
 \begin{align*}
p^{-1}\sum_{\theta_i<\theta_j} P(\hat{\rho}(i)>\hat{\rho}(j))&\geqslant p^{-1}\sum_{\stackrel{\mu+\frac{3}{4}\zeta_p<\theta_i<\mu+\frac{5}{6}\zeta_p}{\mu+\frac{11}{12}\zeta_p<\theta_j<\mu+\zeta_p}} P(\hat{\rho}(i)>\hat{\rho}(j))\\
 &\geqslant p^{-1}\left( p^2P\left(\mu+\frac{3}{4}\zeta_p<\Theta_i<\mu+\frac{5}{6}\zeta_p\right)P\left(\mu+\frac{11}{12}\zeta_p<\Theta_j<\mu+\zeta_p\right)m \right)\\
&\geqslant m
\end{align*}
which will prove the inconsistency.

We will show that setting $\zeta_p$ to be the largest
solution to
$${\zeta_p}^4e^{-\frac{{\zeta_p}^2}{8\tau^2}}=\frac{8192}{2023}e^{-2}{s_{p}}^4\tau^{-2}$$
satisfies the required conditions. Clearly $\zeta_p\rightarrow\infty$, so
for large enough $p$, we have $\frac{\mu}{\zeta_p}<\frac{1}{12}$. The
probabilities in Condition~\eqref{GeneralCounterExampleProbCond} are
calculated from the Pareto distribution:
$$P\left(\mu+\frac{3}{4}\zeta_p<\Theta_i<\mu+\frac{5}{6}\zeta_p\right)=\left(\mu+\frac{3}{4}\zeta_p\right)^{-4}-\left(\mu+\frac{5}{6}\zeta_p\right)^{-4}>\left(\left(\frac{3}{4}+\frac{1}{12}\right)^{-4}-\left(\frac{5}{6}+\frac{1}{12}\right)^{-4}\right){\zeta_p}^{-4}$$
and similarly
$$P\left(\mu+\frac{11}{12}\zeta_p<\Theta_i<\mu+\zeta_p\right)>\left(\left(\frac{11}{12}+\frac{1}{12}\right)^{-4}-\left(1+\frac{1}{12}\right)^{-4}\right){\zeta_p}^{-4}$$
which means that Condition~\eqref{GeneralCounterExampleProbCond} will
hold whenever ${\zeta_p}^{-4}>p^{-\frac{1}{2}}$, or equivalently
$\zeta_p<p^{\frac{1}{8}}$. Since
${\zeta_p}^4e^{-\frac{{\zeta_p}^2}{8\tau^2}}\rightarrow 0$ as
$p\rightarrow\infty$, to show that $p^{\frac{1}{8}}$ is larger than
the largest solution to 
$${\zeta_p}^4e^{-\frac{{\zeta_p}^2}{8\tau^2}}=\frac{8192}{2023}e^{-2}{s_{p}}^4\tau^{-2}$$
it is sufficient to prove that
$$
\left(p^{\frac{1}{8}}\right)^4e^{-\frac{\left(p^{\frac{1}{8}}\right)^2}{8\tau^2}}\leqslant\frac{8192}{2023}e^{-2}{s_{p}}^{4}\tau^{-2}$$
We have that
\begin{align*}
  \left(p^{\frac{1}{8}}\right)^4e^{-\frac{\left(p^{\frac{1}{8}}\right)^2}{8\tau^2}}{s_{p}}^{-4}&=p^{\frac{1}{2}}{s_{p}}^{-4}e^{-\frac{p^{\frac{1}{4}}}{8\tau^2}}\\
  &\leqslant Cp^{\frac{1}{2}}p^{4\alpha}e^{-\frac{p^{\frac{1}{4}}}{8\tau^2}}\\
  &\rightarrow 0
\end{align*}
Therefore, the condition must hold for sufficiently large $p$.

For Condition~\eqref{GeneralCounterExampleMisrankCond}, let
$\mu+\frac{3}{4}\zeta_p<\theta_i<\mu+\frac{5}{6}\zeta_p$ and
$\mu+\frac{11}{12}\zeta_p<\theta_j<\mu+\zeta_p$.  There is probability
at least $\frac{m}{2}\times\frac{m}{2}\times
\left(e^{-1}-e^{-2}\right)$ that $\sigma_{ip}=0$,
$x_i>\mu+\frac{3}{4}\zeta_p$, $x_j<\mu+\zeta_p$ and
$v_p<\sigma_{jp}<2v_p$.  We want to use Lemma~\ref{PMBound} to show
that under these conditions, the units $u_i$ and $u_j$ must be
misranked. It is easy to see that the condition
$\hat{\rho}(i)<\hat{\rho}(j)$, will be retained by increasing $x_i$ or
decreasing $x_j$, since the posterior mean is clearly an increasing
function of $x_i$ and $x_j$. It is therefore sufficient to show that
the units are misranked when $x_i=\mu+\frac{3}{4}\zeta_p$ and
$x_j=\mu+\zeta_p+2\sigma_{jp}$. By Lemma~\ref{PMBound}, it is
sufficient to show that
$$\frac{x_j+\mu}{2}+\frac{289}{4096}e^{2}\tau^2{\sigma_{jp}}^{-4}(x_j-\mu)^{5}e^{-\frac{(x_j-\mu)^2}{8\tau^2}}<x_i$$
Making the above substitutions, we get 
\begin{align*}
\mu+\frac{\zeta_p}{2}+\frac{289}{4096}e^{2}\tau^2{\sigma_{jp}}^{-4}(\zeta_p)^{5}e^{-\frac{(\zeta_p)^2}{8\tau^2}}&\leqslant
\mu+(\zeta_p)\left(\frac{1}{2}+\frac{289}{4096}e^{2}\tau^2{\sigma_{jp}}^{-4}\left(\frac{8192}{2023}e^{-2}{s_{p}}^4\tau^{-2}\right)\right)\\
&\leqslant
\mu+\zeta_p\left(\frac{1}{2}+\frac{2}{7}{s_{p}}^4{\sigma_{jp}}^{-4}\right)\\
&\leqslant \mu+\frac{9}{14}\zeta_p\\
&\leqslant x_i
\end{align*}
so the units are misranked under these conditions.
\end{proof}

\subsection{Normal Error}

\begin{lemma}
Suppose that $X_{ip}$ is normally distributed with mean $\Theta_i$ and
variance ${\sigma_{ip}}^2<1$. Suppose that we estimate ranking using a
prior distribution for $\Theta$ with density $\pi(x)=e^{-e^{\frac{x^2}{4}}}$. For
any $x$ satisfying $x^2\geqslant 4\log 3-8\log(\sigma)$, and
$2\sqrt{\pi}x^2\sigma^2<1$, the posterior mean $\mu$ satisfies
$\mu<x-\frac{1}{x}$.
\end{lemma}

\begin{proof}
Since $\log(\sigma)<0$, we have $x^2\geqslant 4\log
3-8\log(\sigma)>4$, which means that $x>2$. The posterior density of
$\Theta|X$ is given by $$f_{\Theta|X}(\theta)\propto
e^{-e^{\frac{\theta^2}{4}}}e^{-\frac{(\theta-x)^2}{2\sigma^2}}$$ The
posterior mean $\mu$ is the solution to
\begin{align*}
  \int_{-\infty}^{\infty}(\theta-\mu)e^{-e^{\frac{\theta^2}{4}}}e^{-\frac{(\theta-x)^2}{2\sigma^2}}\,d\theta&=0\\
\end{align*}
Thus, it is sufficient to prove that 
$$\int_{-\infty}^{\infty}\left(\theta-\left(x-\frac{1}{x}\right)\right)e^{-e^{\frac{\theta^2}{4}}}e^{-\frac{(\theta-x)^2}{2\sigma^2}}\,d\theta\leqslant0$$
By splitting this integral into the positive and negative part, we get
\begin{align*}
  \int_{-\infty}^{\infty}\left(\theta-\left(x-\frac{1}{x}\right)\right)e^{-e^{\frac{\theta^2}{4}}}e^{-\frac{(\theta-x)^2}{2\sigma^2}}\,d\theta&\leqslant
  -\int_{x-\frac{3}{x}}^{x-\frac{2}{x}}\frac{1}{x}e^{-e^{\frac{\left(x-\frac{2}{x}\right)^2}{4}}}e^{-\frac{(\theta-x)^2}{2\sigma^2}}\,d\theta+\int_{x-\frac{1}{x}}^\infty\left(\theta-\left(x-\frac{1}{x}\right)\right)e^{-e^{\frac{\theta^2}{4}}}e^{-\frac{(\theta-x)^2}{2\sigma^2}}\,d\theta\\
  &\leqslant
  -\frac{e^{-e^{\frac{\left(x-\frac{2}{x}\right)^2}{4}}}}{x}\int_{x-\frac{3}{x}}^{x-\frac{2}{x}}e^{-\frac{(\theta-x)^2}{2\sigma^2}}\,d\theta+e^{-e^{\frac{\left(x-\frac{1}{x}\right)^2}{4}}}\int_{x-\frac{1}{x}}^\infty\left(\theta-x+\frac{1}{x}\right)e^{-\frac{(\theta-x)^2}{2\sigma^2}}\,d\theta\\
  &\leqslant
  -\frac{e^{-e^{\frac{\left(x-\frac{2}{x}\right)^2}{4}}}}{x^2}e^{-\frac{9}{2x^2\sigma^2}}+e^{-e^{\frac{\left(x-\frac{1}{x}\right)^2}{4}}}\left(\sqrt{2\pi}\sigma{\mathbb
    E}_{\Theta\sim N(x,\sigma^2)}\left((\Theta-x)_++\frac{1}{x}\right)\right)\\
  &=
  -\frac{e^{-e^{\frac{\left(x-\frac{2}{x}\right)^2}{4}}}}{x^2}e^{-\frac{9}{2x^2\sigma^2}}+\sqrt{2\pi}\left(\sqrt{2}\sigma^2+\frac{\sigma}{x}\right)e^{-e^{\frac{\left(x-\frac{1}{x}\right)^2}{4}}}\\
\end{align*}

Since $x>2$, we have $$e^{\frac{1}{4x^2}-\frac{1}{2}}\left(1-e^{-\frac{3}{2}+\frac{3}{4x^2}}\right)\geqslant
e^{-\frac{1}{2}}\left(1-e^{-1}\right)\geqslant \frac{3}{8}$$
and since $2\sqrt{\pi}x^2\sigma^2<1$, we have
$\log(2\sqrt{\pi})+2\log(x\sigma)<0$ and $x^2\geqslant
4\log(3)-8\log(\sigma)$. Thus, we get
\begin{align*}
  e^{\frac{x^2}{4}}&\geqslant \frac{3}{\sigma^2}\\
  e^{\frac{1}{4x^2}-\frac{1}{2}}\left(1-e^{-\frac{3}{2}+\frac{3}{4x^2}}\right)e^{\frac{x^2}{4}}&\geqslant
  \frac{3}{8}e^{\frac{x^2}{4}}\geqslant \frac{9}{8\sigma^2}\geqslant \frac{9}{2x^2\sigma^2}+\log(2\sqrt{\pi})+2\log(x\sigma)\\
  e^{e^{\frac{1}{4}\left(x-\frac{1}{x}\right)^2}-e^{\frac{1}{4}\left(x-\frac{2}{x}\right)^2}-\frac{9}{2x^2\sigma^2}}&\geqslant
2\sqrt{\pi}x^2\sigma^2\\
-\frac{e^{-e^{\frac{\left(x-\frac{2}{x}\right)^2}{4}}}}{x^2}e^{-\frac{9}{2x^2\sigma^2}}+2\sqrt{\pi}\sigma^2e^{-e^{\frac{\left(x-\frac{1}{x}\right)^2}{4}}}&\leqslant
0
\end{align*}

%
%
%
%
\end{proof}

Thus, any pair of units $u_i$, $u_j$ satisfying the following
constraints will be misranked by posterior mean:

\begin{itemize}

\item $\sigma_{ip}=0$, $v_p<\sigma_{jp}<1$.

\item $\theta_i<\theta_j$ and $x_i>x_j-\frac{1}{x_j}$.


\item ${x_j}^2\geqslant 4\log 3-8\log(\sigma_{jp})$
  and $2\sqrt{\pi}{x_j}^2{\sigma_{jp}}^2<1$.
  
\end{itemize}

\begin{theorem}
  Let $\Theta_i$ be independently drawn from a true prior with density
  $f(\theta)=\frac{|\Theta|e^{-|\theta|}}{2}$, and $X_{ip}$ be
  independantly normally distributed with mean $\Theta_i$ and variance
  ${\sigma_{ip}}^2$, where ${\sigma_{ip}}^2$ are i.i.d.  with
  probability mass $\frac{1}{2}$ at 0 and non-zero values following an
  exponential distribution with mean $v_p$. Suppose $v_p$ satisfies
  $2{\sqrt{-8\log(v_p)}}=\log(p)$. Then posterior mean ranking of
  units is inconsistent with respect to total misranking loss.
\end{theorem}

\begin{proof}
  It is sufficient to show that $p^2P(\Theta_i>\Theta_j\textrm{ and }\hat{\rho}(i)<\hat{\rho}(j))\not\rightarrow 0$, i.e. the expected number of
  misranked pairs does not converge to 0. We have found sufficient
  conditions to ensure that a pair are misranked. We know that
  $P(\sigma_{ip}=0\textrm{ and }
  v_p<\sigma_{jp}<2v_p)=\frac{e^{-1}-e^{-2}}{4}$.
Thus we only need to ensure
that $$p^2P\left(\left({x_j}^2\geqslant 4\log
3-8\log(\sigma_{jp})\right)\textrm{ and }\left(x_j-\frac{1}{x_j}<x_i<x_j\right)\textrm{ and }\left({x_j}^2<\frac{1}{2\sqrt{\pi}{\sigma_{j}}^2}\right)\right)\not\rightarrow 0$$ 
Since $\sigma_{jp}<2v_p$, we have
$P\left(|x_{jp}-\theta_j|<4v_p\right)>\frac{3}{4}$ by Chebyshev's
inequality. Thus, if
$$\theta_{j}+4v_p-\frac{1}{\theta_{j}+4v_p}<\theta_{i}<\theta_{j}-4v_p$$ the condition
$x_{jp}-\frac{1}{x_{jp}}<x_{ip}<x_{jp}$ will hold with probability
at least $\frac{3}{4}$. For any $\theta_j$ satisfying $\theta_j+\frac{1}{\theta_j+4v_p}-4v_p>1$, we have
$$P\left(\theta_{j}+4v_p-\frac{1}{\theta_{j}+4v_p}<\Theta_{i}<\theta_{j}-4v_p\right)=\int_{\theta_{j}+4v_p-\frac{1}{\theta_{j}+4v_p}}^{\theta_{j}-4v_p}|\theta|e^{-|\theta|}\,d\theta\geqslant \left(\frac{1}{\theta_j+4v_p}-8v_p\right)\theta_je^{-\theta_j}$$
where the inequality is because $\theta e^{-\theta}$ is a decreasing
function of $\theta$ for $\theta>1$. Therefore, 
\begin{align*}
  P\left(\left(\Theta_{j}+4v_p-\frac{1}{\Theta_{j}+4v_p}<\Theta_{i}<\Theta_{j}-4v_p\right)\vphantom{\frac{1}{\sqrt{2\sqrt{\pi}}v_p}}\right.&\left.\land\left(\frac{1}{\sqrt{2\sqrt{\pi}}v_p}>\Theta_j>\sqrt{-8\log(v_p)}\right)\right)\\
  &\geqslant\int_{\sqrt{-8\log(v_p)}}^{\frac{1}{\sqrt{2\sqrt{\pi}}v_p}}
\theta^2e^{-2\theta}\left(\frac{1}{\theta+4v_p}-8v_p\right)_+\,d\theta\\
&\geqslant \left(\frac{1}{2}-32v_p\right) \int_{\sqrt{-8\log(v_p)}}^{\frac{1}{16v_p}-4v_p}
\theta e^{-2\theta}\,d\theta\\
&=\left(\frac{1}{2}-32v_p\right)\left(\left[-\theta \frac{e^{-2\theta}}{2}\right]_{\sqrt{-8\log(v_p)}}^{\frac{1}{16v_p}-4v_p}+\int_{\sqrt{-8\log(v_p)}}^{\frac{1}{16v_p}-4v_p}
\frac{e^{-2\theta}}{2}\,d\theta\right)\\
&=\left(\sqrt{-8\log(v_p)}+\frac{1}{2}\right)e^{-2{\sqrt{-8\log(v_p)}}}-\left(\frac{1}{16v_p}-4v_p+\frac{1}{2}\right)e^{-\frac{1}{8v_p}+8v_p}\\
&=\sqrt{-8\log(v_p)}e^{-2{\sqrt{-8\log(v_p)}}}+o\left(\sqrt{-8\log(v_p)}e^{-2{\sqrt{-8\log(v_p)}}}\right)\\
&=\frac{\log(p)}{2}e^{-\log(p)}+o\left(\frac{\log(p)}{2}e^{-\log(p)}\right)\\
&=O\left(\frac{\log(p)}{p}\right)
\end{align*}
Thus 
$p^2P(\Theta_i>\Theta_j\textrm{ and
}\hat{\rho}(i)<\hat{\rho}(j))\geqslant O(p\log(p))$ so
posterior mean ranking is inconsistent.
\end{proof}

\section{Conclusions}\label{Conclusions}

We have shown that a large class of ranking methods are consistent
provided $p^2s(p){\mathbb
  E}\left({\sigma_{ip}}\right)^{\frac{1}{3}}\rightarrow 0$.
In cases where $\sigma_{ip}$ is the standard error of an estimator
from a sample of size $n$, we will usually have
$\sigma_{ip}\propto\frac{1}{n}$, meaning that ranking methods are
total loss convergent provided $p^2n^{-\frac{1}{3}}\rightarrow
0$. This condition for consistency is stricter than for the
consistency of continuous estimators, but it does apply in
misspecified cases, and the number of pairwise rankings to be
considered inceases in proportion to $p^2$, so consistency of ranking
methods is a stricter requirement than consistency of parameter estimators.

The key condition for consistency of ranking methods is that the prior
distribution be at least as heavy-tailed as the error
distribution. This requirement typically holds when we use a conjugate
prior, which is common practice. We have provided examples where
light-tailed priors lead to inconsistent ranking estimators. Our
results are based on the assumption that both the prior and the loss
function could be misspecified.

While we have shown that ranking methods are all consistent, even in
the event that they are misspecified, there is a lot more work that
could be done in terms of studying the asymptotic and finite sample
behaviour of various ranking methods. Research (Kenney, He and Gu,
2016) suggests that methods with heavy-tailed prior distributions are
more robust to model misspecification. This may manifest itself in the
form of faster convergence guarantees for these methods, or better
finite sample performance. These issues will be studied in further
papers.

\end{document}